%% file: paperMgxfeFinal.tex
\documentclass{siamltex}
\usepackage{amsfonts}
\usepackage{amsmath}
\usepackage{mathtools}
\usepackage{latexsym}
\usepackage{amssymb}
\usepackage{booktabs}
\usepackage{array}

\usepackage{graphicx,overpic,xcolor}
\usepackage[list=true, 
labelformat=brace, position=top]{subcaption}
\graphicspath{{./pictures/}{./asymptote/}}
\DeclareGraphicsExtensions{.png,.pdf}

\usepackage[english=american]{csquotes}
\usepackage{hyperref}
\hypersetup{
    colorlinks=true, 
    linktoc=all,     
    linkcolor=blue,  
    citecolor=red,
}
\usepackage{algorithm}
\usepackage{algpseudocode}
\usepackage{setspace}
\usepackage{pgfplots}
\pgfplotsset{compat=newest}
\usetikzlibrary{patterns}
\definecolor{rwthBlue}{RGB}{0,84,159}
\definecolor{rwthLightBlue}{RGB}{142,186,229}
\definecolor{rwthOrange}{RGB}{246,168,0}
\definecolor{rwthGreen}{RGB}{87,171,39}
\definecolor{rwthDarkGreen}{RGB}{0,97,101}
\definecolor{rwthLightGreen}{RGB}{189,205,0}
\definecolor{rwthCyan}{RGB}{0,152,161}
\definecolor{rwthRed}{RGB}{204,7,30}
\definecolor{rwthDarkRed}{RGB}{161,16,53}
\definecolor{rwthLila}{RGB}{122,111,172}
\definecolor{rwthPurple}{RGB}{97,33,88}
\definecolor{rwthYellow}{RGB}{255,237,0}
\definecolor{IGPMred}{RGB}{156,12,16}
\definecolor{IGPMblue}{RGB}{0,105,164}

\newcounter{assum}
\renewcommand*{\theassum}{A\arabic{assum}}

\newcommand{\bu} {\mathbf{u}}
\newcommand{\bw} {\mathbf{w}}

\newcommand{\bp} {\mathbf{p}}
\newcommand{\br} {\mathbf{r}}

\newcommand{\bA} {\mathrm{A}}

\newcommand{\bM} {\mathrm{M}}
\newcommand{\bD} {\mathrm{D}}

\newcommand{\la} {\langle}
\newcommand{\ra} {\rangle}

\newcommand{\cW} {{\cal W}}

\newcommand{\cV} {{\cal V}}
\newcommand{\cIG}{\mathcal{I}^\Gamma}
\newcommand{\R} {\mathbb{R}}
\newcommand{\T} {\mathcal{T}}

\newcommand{\bv} {\mathbf{v}}

\newtheorem{remark}{Remark}
\newtheorem{assumption}{Assumption}
\newcommand{\jumpleft}{[\![}
\newcommand{\jumpright}{]\!]}
\newcommand{\jump}[1]{\jumpleft #1 \jumpright}
\newcommand{\vv}[1]{\mathbf{#1}}        

\newcommand{\BrokenOmega}{\Omega_{1}\cup \Omega_2}

\newcommand{\averageleft}{\{\!\!\{}
\newcommand{\averageright}{\!\}\!\!\}}
\newcommand{\average}[1]{\averageleft \! #1 \averageright}

\newcommand{\vertiii}[1]{{\left\vert\kern-0.25ex\left\vert\kern-0.25ex\left\vert #1 
    \right\vert\kern-0.25ex\right\vert\kern-0.25ex\right\vert}}

\newcommand{\cR}{\ensuremath{\mathcal{R}}}
\newcommand{\xferest}{\ensuremath{\mathcal{R}}}

\newcommand{\mone}{\ensuremath{{-1}}}
\newcommand{\lmone}{\ensuremath{{\ell-1}}}

\newcommand{\Pellg}{\ensuremath{P_\ell^\Gamma}}

\newcommand{\uellg}{\ensuremath{u_{h_\ell}}}

\newcommand{\Vellg}{\ensuremath{V_\ell^\Gamma}}
\newcommand{\Vellmg}{\ensuremath{V_{\ell-1}^\Gamma}}
\newcommand{\Vellgl}{\ensuremath{V_\ell^{\Gamma_\ell}}}
\newcommand{\Vellmgl}{\ensuremath{V_{\ell-1}^{\Gamma_\lmone}}}
\newcommand{\Velli}[1]{\ensuremath{V_{\ell,#1}}}
\newcommand{\Vellmi}[1]{\ensuremath{V_{\ell-1,#1}}}
\newcommand{\Omegaie}{\ensuremath{\Omega^e_{\ell,i}}}

\newcommand{\phielli}{\ensuremath{\varphi^k_{\ell,i}}}

\newcommand{\Tell}{\ensuremath{\mathcal{T}_{\ell}}}
\newcommand{\Tellg}{\ensuremath{\mathcal{T}^\Gamma_{\ell}}}

\newcommand{\Iellg}{\ensuremath{I_\ell^\Gamma}}
\newcommand{\Iellgi}[1]{\ensuremath{I_{\ell,#1}^\Gamma}}

\def\b|{{\|\hskip-0.16ex|}}

\title{A multigrid method for unfitted finite element discretizations of  elliptic interface problems}

\author{Thomas Ludescher\thanks{Institut f\"ur
Geometrie und Praktische  Mathematik, RWTH Aachen University, D-52056 Aachen,
Germany; email: {\tt ludescher@igpm.rwth-aachen.de}}
\and Sven Gross\thanks{Institut f\"ur
Geometrie und Praktische  Mathematik, RWTH Aachen University, D-52056 Aachen,
Germany; email: {\tt gross@igpm.rwth-aachen.de}} 
\and Arnold Reusken\thanks{Institut f\"ur
Geometrie und Praktische  Mathematik, RWTH Aachen University, D-52056 Aachen,
Germany; email: {\tt reusken@igpm.rwth-aachen.de}} }

\begin{document}
{

\maketitle

\begin{abstract}
We consider discrete  Poisson interface problems resulting from  linear unfitted finite elements, also called cut finite elements (CutFEM). Three of these unfitted finite element  methods known from the literature are studied. All three methods  rely on Nitsche's method to incorporate the interface conditions. The main topic of the paper is the development of a multigrid method, based on a novel prolongation operator for the unfitted finite element space and an interface smoother that is designed to yield robustness for large jumps in the diffusion coefficients. Numerical results are presented which illustrate efficiency of this multigrid method and demonstrate  its robustness properties with respect to variation of the mesh size, location of the interface and contrast in the diffusion coefficients.
\end{abstract}


\section{Introduction}
In this paper we consider the elliptic interface problem
\begin{subequations}\label{eq:ellipIf}
\begin{align}
 -\nabla \cdot (\mu \nabla u ) &= f &&\text{in } \Omega_i, ~i=1,2 \label{eq:ellipIf1}\\
 u &= u_D &&\text{on } \partial\Omega \\
 \jump{u} &= 0 &&\text{on } \Gamma  \label{eq:ellipIf3}\\
 \jump{-\mu\nabla u \cdot \vv{n} } &= 0 &&\text{on } \Gamma,  \label{eq:ellipIf4}
\end{align}
\end{subequations}
on a polygonal domain $\Omega = \Omega_1 \cup \Gamma \cup \Omega_2 \in \R^d$, $d=2,3$. The subdomains $\Omega_1$ and $\Omega_2$ are separated by the interface $\Gamma$, i.e., $\Gamma =\overline \Omega_1 \cap \overline \Omega_2$.  For simplicity, we assume $\Gamma$ to be  connected. 
The diffusion coefficient $\mu > 0$ is assumed to be  piecewise constant, $\mu(x)=\mu_i=\text{const}$ on $\Omega_i$, with $\mu_1 \neq \mu_2$. The solution of problem \eqref{eq:ellipIf} then exhibits a kink at the interface $\Gamma$. 
For a mesh which is \emph{not} aligned to the interface, standard finite element methods can not approximate  the solution $u$ accurately. Starting with the method introduced in \cite{Hansbo02}, based on Nitsche's method to incorporate the interface conditions \cite{nitsche1971variationsprinzip}, several \emph{un}fitted finite element methods, also called CutFEM, have been studied. These methods yield an optimal 
order of convergence for this class of elliptic interface problems. We refer to the recent overview paper \cite{burman2015cutfem} and the references therein for a detailed discussion of these methods. 

In this paper we introduce and analyze an efficient multigrid method for a class of unfitted finite element methods. We restrict to piecewise linear finite elements. Although several different unfitted finite element methods have been developed (cf. \cite{burman2015cutfem}), almost all of these methods use the same unfitted finite element space $V_\ell^\Gamma$, defined in \eqref{eq:vgammafull}. From the class of methods using this space we consider the following  three important representatives. Firstly, as discretization method we use the original Nitsche unfitted finite element method \cite{Hansbo02}, which was the starting point for the development of several variants. In this method one has to select an appropriate value for a stabilization parameter. The second unfitted finite element Nitsche type method that we consider, introduced in  \cite{lehrenfeld2016removing}, is parameter-free. These first two methods are not robust with respect to large jumps in the diffusion coefficient $\mu$. Therefore, as 
third representative we use the method from \cite{burman2011numerical}, which has  strong robustness properties w.r.t. variation of the coefficient jump $\mu_1/\mu_2$. 
This third method uses a so-called   ghost penalty stabilization \cite{Burman}, which again requires the selection of a stabilization parameter value. 
We mention two further papers \cite{annavarapu2012robust,barrau2012robust} in which  variants of these unfitted finite element methods for interface problems are studied.  In these papers  a particular choice of weights for the averaging operator (used in the Nitsche method) is introduced which leads to very strong robustness properties both  with respect to large jumps in $\mu$  and the interface location. It turns out that for certain cut configurations and  specific diffusion coefficient ratios  the Nitsche penalty parameter becomes arbitrarily large (causing a corresponding condition number blow up for the stiffness matrix).  This is not the case for the method from \cite{Burman}, and therefore we focus on that method.

Unfitted finite element methods are rather popular in fracture mechanics \cite{fries2010extended} (in that community these are often called extended finite element methods, XFEM). In \cite{berger2012inexact,gerstenberger2013algebraic},  algebraic multigrid methods for that  problem class are developed. The performance of these methods for the interface problem \eqref{eq:ellipIf3}-\eqref{eq:ellipIf4} has not been studied. 

It is well known (cf.~\cite{burman2015cutfem}) that the stiffness matrices resulting from such unfitted finite element discretization methods may have a blow up in the condition number due to an instability caused by arbitrarily small cuts. This blow up effect can be avoided by a ghost penalty stabilization. As known from the literature, the methods from \cite{Hansbo02,lehrenfeld2016removing} may suffer from this blow up effect, whereas (due to the ghost penalty stabilization) the method from  \cite{ burman2011numerical} does not have this drawback. As we will see in this paper, the multigrid method that we propose does not deteriorate if the condition number blows up due to small cuts. 

There are only very  few papers that we know of in which efficient iterative solvers for the discrete interface problems resulting from  unfitted finite element methods  are treated.
In \cite{lehrenfeld2017optimal} a preconditioner for the original Nitsche unfitted finite element method \cite{Hansbo02} for \eqref{eq:ellipIf} is developed which relies on a stable subspace splitting of the space $V_\ell^\Gamma$. This space is split into the standard linear finite element space and a complement spanned by the  cut basis functions.  A block diagonal preconditioner is proposed that leads to a bounded condition number uniformly in the mesh size and the location of the interface. The condition number, however, deteriorates for larger jumps in the diffusion coefficient.  

There are several papers in which multigrid methods  are presented for discretizations of \eqref{eq:ellipIf} that differ from  the unfitted finite element methods considered in this paper. We mention a few of these. 
In \cite{adams2002immersed,adams2004new} a multigrid solver for  the immersed interface  method (IIM) \cite{leveque1994immersed} is presented, which  uses nine point stencils
on regular grids and a modified interpolation. A Krylov-accelerated multigrid method combined with the IIM is presented in \cite{chen2008piecewise}.   In \cite{li1998fast} a preconditioner for the IIM is developed that makes use of fast Poisson solvers. Recently, in  \cite{jo2018geometric} a geometric multigrid method has been studied for a finite element variant (on structured grids) of the IIM.
For the special and simpler case, where the interface is aligned with the grid, a multigrid method with a specific coarsening strategy, such that the interface is aligned with coarse grids as well, is proposed in \cite{wan2000interface}.  A multilevel preconditioner for a discontinuous Galerkin discretization of large contrast elliptic interface  problems, with an interface aligned to the grid, is presented in \cite{ayuso2014multilevel}. 

There are a few papers in which fast  solvers for unfitted finite element methods in the setting of fictituous domain techniques are treated, e.g. \cite{badia2017robust,de2017preconditioning}.

To our knowledge an efficient  geometric multigrid method for a class of unfitted finite element discretizations of \eqref{eq:ellipIf} has not been treated in the literature. In this paper we present  such a method. In this method there are two key ingredients that make it different from a standard multigrid method (for the Poisson equation).
The first one is  a novel prolongation operator that is tailor-made for the non-nested hierarchy of unfitted finite element spaces. Furthermore, for the case with large jumps in the diffusion  coefficient we  propose a smoothing procedure based on  a combination of a standard Gauss-Seidel smoother with a specific local (close to the interface) correction procedure. This then results in a multigrid solver which has arithmetic costs per iteration  comparable to a few matrix-vector multiplications and a contraction number (much) smaller than one uniformly in the mesh size, the location of the interface and the jump in the diffusion coefficient (where the latter property holds only for the robust discretization method from \cite{ burman2011numerical}). This optimality property is derived from numerical experiments for the three unfitted finite element discretizations outlined above. 
  
The paper is organized as follows.  In section \ref{sec:disc} we introduce the three discretization methods from \cite{Hansbo02,lehrenfeld2016removing,burman2011numerical} that are considered. In section~\ref{sectMG} we explain the multigrid method. In particular the new prolongation operator and a special (robust) smoother are introduced. The performance of the proposed multigrid method for the different unfitted finite element discretization methods  is studied  in section \ref{sec:exp}.

\section{Unfitted finite element discretizations} \label{sec:disc}
In this section we describe three known unfitted finite element methods. These methods are the original Nitsche based CutFEM presented in \cite{Hansbo02} (section~\ref{HansboMeth}), a variant without Nitsche stabilization parameter from \cite{lehrenfeld2016removing}  (section~\ref{SectLehrenfeld}), and a method from \cite{burman2011numerical} which has better robustness properties w.r.t. the jump in the diffusion coefficient $\mu$ across the interface. In the first two methods we do \emph{not} use any stabilization (e.g., ghost penalty)  for improving the in general very poor conditioning of the system matrix. In the  third method a ghost penalty stabilization is applied.   All three  methods use the same unfitted finite element space $V_\ell^\Gamma$, which is  introduced below. We discuss certain elementary properties of the corresponding finite element isomorphism, which maps an element of $V_\ell^\Gamma$ to its representation in a particular basis of $V_\ell^\Gamma$. These properties are relevant for the derivation of the prolongation and restriction operators in the multigrid method in section~\ref{sectMG}.   

Let $\{\mathcal{T}_\ell\}$, $\ell=0,1,\ldots$  be a family of nested, quasi-uniform simplicial triangulations of $\Omega$ which are not fitted to the interface $\Gamma$ ($\ell=0$ corresponds to the coarsest level). The mesh size parameter corresponding to $\mathcal{T}_\ell$ is denoted by $h_\ell$. We introduce the standard space of continuous 
piecewise linear finite elements with zero boundary conditions 
\begin{equation} \label{DefVl}
 V_\ell:= \{\, v_h \in C(\Omega)~|~{v_h}_{|\partial \Omega}=0,~~~{v_h}_{|T} \in \mathcal{P}_1\quad\text{for all}~~T \in \T_\ell\,\}.
\end{equation}
The unfitted variant of the space $V_\ell$ is obtained by ``cutting finite element functions along $\Gamma$''. In practice an approximation $\Gamma_\ell \approx \Gamma$ is used, which has a level dependent accuracy. In a multigrid solver one then has to deal with  level dependent interface approximations $\Gamma_\ell$, where the approximation on a coarse level  may be significantly different from the one on a fine level. This has to be taken into account in the prolongation and restriction operators used in the multigrid method. Actually, this is a key point in the particular multigrid method that we introduce in section~\ref{sectMG}.  We assume  approximations $\Gamma_\ell \approx \Gamma$, $\ell=0,1,2, \ldots,$ where $\Gamma_\ell$ is a connected interface approximation and define the set of cut elements $\Tellg := \{ T \in \Tell \mid \operatorname{meas}_{d-1}( T \cap \Gamma_\ell) > 0 \}$.  We define a corresponding partitioning of the domain $\Omega= \Omega_{\ell,1} \cup \Gamma_\ell \cup \Omega_{\ell,2}$, with cut subdomains $\Omega_{\ell,i}$ and  $\Gamma_\ell =\overline \Omega_{\ell,1} \cap \overline \Omega_{\ell,2}$; cf. Figure~\ref{fig:ext_domains}.
Using the  restriction operators to the subdomains $\Omega_{\ell,i}$ 
\begin{equation} \label{eq:restop}
 \xferest_{\ell,i} v = \begin{cases}
          v_{|\Omega_{\ell,i}} &\text{on } \Omega_{\ell,i} \\
          0            &\text{on } \Omega \setminus \Omega_{\ell,i}
         \end{cases}, ~~~i=1,2,
\end{equation}
we define the \emph{unfitted finite element space}
\begin{equation} \label{eq:vgammafull}
\Vellg = \xferest_{\ell,1} V_\ell \oplus \xferest_{\ell,2} V_\ell.
\end{equation}
Note that in general, due to $\Gamma_{\ell-1} \neq \Gamma_\ell$, these spaces are \emph{not nested}.
For a description of the (nodal) basis in this space and the corresponding finite element isomorphism (mapping functions to the vector of coefficients in the basis) it is convenient to use an isomorphism between this unfitted finite element space $\Vellg $ and another finite element space; cf. \eqref{iso} below. For this we first need some further notation.  
We introduce the \emph{extended}  domains  $\Omegaie \supset \Omega_{\ell,i}$
\begin{align*}
\Omegaie &:= \cup \, \{\, T \in \mathcal{T}_\ell ~\mid ~ T \cap \Omega_{\ell,i} \neq \emptyset \, \}, \quad i=1,2, 
\end{align*}
(cf. Figure~\ref{fig:ext_domains} for illustration).

\begin{figure}[ht!]
\centering
\begin{subfigure}[c]{0.45\textwidth}
\centering
\begin{overpic}[width=0.8\textwidth]{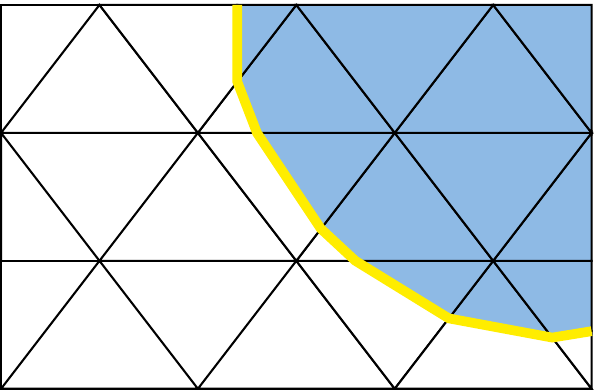}
  \put(105,10){\color{rwthPurple}$\Gamma_\ell$}
  \put(78,50){$\Omega_{\ell,1}$}
\end{overpic}
\subcaption{Cut subdomain $\Omega_{\ell,1}$.}
\end{subfigure}
\begin{subfigure}[c]{0.45\textwidth}
\centering
\begin{overpic}[width=0.8\textwidth]{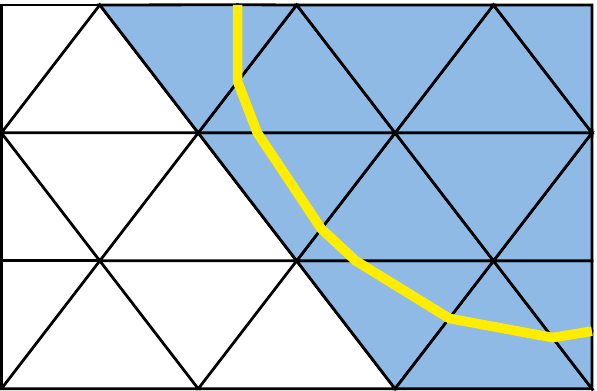}
  \put(105,10){\color{rwthPurple}$\Gamma_\ell$}
  \put(78,50){$\Omega_{\ell,1}^e$}
\end{overpic}
\subcaption{Extended subdomain $\Omega_{\ell,1}^e$ on $\mathcal{T}_\ell$.}
\end{subfigure}
\caption{Cut and extended subdomains.}
\label{fig:ext_domains}
\end{figure}

We formulate assumptions on the interface approximations $\Gamma_\ell$ and the corresponding extended subdomains $\Omegaie$.
\begin{assumption}\label{ass} For  $\mathcal{T}_\ell$, $\Gamma_\ell$, $\ell=0,1, \ldots$ the following holds:
\begin{enumerate}
\item[$~\refstepcounter{assum}(\theassum)\label{ass1}$] ${\rm dist}(\Gamma, \Gamma_\ell)  \leq c h_\ell^2$, 
\item[$~\refstepcounter{assum}(\theassum)\label{ass2}$] $ \Omegaie  \subseteq \Omega_{\ell-1,i}^e, \quad i=1,2$,  
\item[$~\refstepcounter{assum}(\theassum)\label{ass3}$] For all $T\in \mathcal{T}_\ell$ with $T \cap \Gamma_\ell \neq \emptyset$ (hence, $T \cap \overline{\Omega}_{\ell,i} \neq \emptyset$ for $i=1$ and $2$) this intersection 
 does not coincide with a subsimplex of $T$ 
(a face, edge or vertex of $T$). 
\end{enumerate}
 \end{assumption}
 
\begin{remark} \label{rem1} \rm 
 One particular choice for the approximate interface $\Gamma_\ell$ that is used in most applications of CutFEM is based on piecewise  linear interpolation. If $\Gamma$ is characterized as the zero level of some level set function $\phi$, and $I_\ell \phi$ is the piecewise linear nodal interpolation in the finite element space $V_\ell$, one defines $\Gamma_\ell:=\{\, x \in \Omega~|~(I_\ell\phi)(x)=0\,\}$.
\end{remark}

\begin{remark}
\rm  We briefly discuss Assumption~\ref{ass}. The assumption \eqref{ass3}, which is satisfied in generic cases, is made to simplify the presentation and avoid technical details.  The geometric accuracy assumption \eqref{ass1} is made to guarantee optimal second order discretization error bounds. The   assumption \eqref{ass2} is important for the construction of the multigrid prolongation operator. The assumption formalizes that the coarse level simplices in $\T_{\ell-1}$ that contain $\Gamma_\ell$ are the same as those containing $\Gamma_{\ell-1}$. Or in other words, if $\T_{\ell-1}^{\Gamma}$ denotes the subdomain formed by all simplices (on level $\ell-1$) which are intersected by the interface approximation $\Gamma_{\ell-1}$, then the interface approximation $\Gamma_\ell$ on the next level remains in this subdomain. The assumptions \eqref{ass1} and \eqref{ass2} are satisfied for the example discussed in Remark~\ref{rem1}.   
\end{remark}

 Corresponding to the extended subdomains we define the finite element spaces
\begin{equation} \label{deflocVh} \Velli{i}:= \{\, (v_h)_{|\Omegaie}~|~ v_h \in V_{\ell} \, \}, ~~i=1,2.
\end{equation}
These are standard linear finite element spaces (on the triangulated domain $\Omegaie$). The standard nodal basis in $\Velli{i}$ is denoted by $(\phielli)_{1 \leq k \leq n_{\ell,i}}$ and the corresponding finite element isomorphism is given by
\[
  P_{\ell,i} : \,\R^{n_{\ell,i}} \to \Velli{i}, \quad P_{\ell,i}\bu_i = \sum_{k=1}^{n_{\ell,i}} u_i^k \phielli.
\]
One easily checks that
\begin{equation} \label{iso}
 \cR_\ell:\, \Velli{1}\times \Velli{2} \to \Vellg, \quad \cR_\ell \begin{pmatrix} u_{h,1} \\ u_{h,2}\end{pmatrix} := \cR_{\ell,1} u_{h,1} +\cR_{\ell,2} u_{h,2} 
\end{equation}
defines \emph{an isomorphism}. As a first direct consequence of this we have that
\begin{equation} \label{basis}
  \{\, \cR_{\ell,i} \phielli~|~ 1\leq k \leq n_{\ell,i},~i=1,2 \, \} 
\end{equation}
is a basis of $\Vellg$. This basis is convenient for the derivation of a multigrid prolongation operator. Another useful basis of $\Vellg$ is introduced in section~\ref{sec:smooth}.

The natural finite element  isomorphism $\Pellg : \R^{n_{\ell,1} + n_{\ell,2}} \to \Vellg$ is given by 
\begin{subequations}\label{eq:feisomorph}
\begin{align}
& \Pellg \begin{pmatrix} \bu_1 \\ \bu_2 \end{pmatrix}: =
 \cR_\ell \begin{pmatrix} 
  P_{\ell,1} \bu_1 \\
P_{\ell,2} \bu_2 \end{pmatrix}, \quad \text{hence}
  \label{eq:feisomorpha}\\
 & (\Pellg)^{-1}\cR_\ell \begin{pmatrix} u_{h,1} \\ u_{h,2}\end{pmatrix} =\begin{pmatrix} 
  P_{\ell,1}^{-1} u_{h,1} \\
P_{\ell,2}^{-1} u_{h,2}
 \end{pmatrix} \label{eq:feisomorphb}
\end{align}
\end{subequations}
holds. 
These relations will be useful for the construction of an appropriate multigrid prolongation operator. One easily checks that for $\Pellg $ defined in \eqref{eq:feisomorpha} and $u_h \in \Vellg$, the vector $(\Pellg)^{-1} u_h$ contains the coefficients of the representation of $u_h$ in the nodal basis \eqref{basis}.

\subsection{Nitsche based CutFEM} \label{HansboMeth}
We briefly recall the method introduced  in \cite{Hansbo02}. Below $u_h \in \Vellg$, $(u_{h,1},u_{h,2}) \in \Velli{1}\times \Velli{2}$ are always related by the isomorphism \eqref{iso}. We use the averaging operator (on level $\ell$)
\begin{align*}
 \average{u_h} := ( \kappa_1 u_{h,1} + \kappa_2 u_{h,2} )|_{\Gamma_\ell}, \quad u_h \in \Vellg,
\end{align*}
with $
\kappa_i = \frac{|T_i|}{|T|}
$, where
 $T_i := T \cap \Omega_{\ell,i}$ for $T \in \T_\ell$. 
The jump operator (on level $\ell$) is defined by 
\begin{align*}
\jump{u_h} := (u_{h,1} - u_{h,2})|_{\Gamma_\ell}, \quad u_h \in \Vellg.
\end{align*}
The Nitsche CutFEM for discretization of \eqref{eq:ellipIf} is as follows: determine $u_h \in  \Vellg $ such that
\begin{equation}\label{eq:nitscheBfa}
a_h( u_h, v_h ) = l(v_h):= (f,v_h)_\Omega \quad \text{for all}~~ v_h \in \Vellg, \tag{Nitsche}
\end{equation}
with 
\begin{align} 
a_h( u, v) & = a(u,v) + N^c(u,v) + N^c(v,u) + N^s_{\lambda_N}(u,v),\\
a(u,v) &:= (\mu \nabla u, \nabla v)_{\Omega_{\ell,1} \cup \Omega_{\ell,2}}, \label{N1}\\
N^c(u,v) &:= ( \average{ - \mu \nabla u \cdot \vv{n}_\ell }, \jump{v} )_{\Gamma_\ell}, \\
N^s_{\lambda_N}(u,v) &:= \frac{\lambda_N}{h_\ell} ( \jump{u}, \jump{v} )_{\Gamma_\ell}. \label{N4}
\end{align}
Note that the bilinear forms $a(\cdot,\cdot)$, $N^c(\cdot,\cdot)$, $N^s_{\lambda_N}(\cdot,\cdot)$ depend on the level $\ell$ through the level dependent approximation $\Gamma_\ell \approx \Gamma$. The vector $ \vv{n}_\ell $ is the normal vector on $\Gamma_\ell$ (defined a.e.).
Symmetry of the bilinear form  in  \eqref{eq:nitscheBfa} is clear and  consistency, continuity and coercivity are proved in \cite{Hansbo02}. The coercivity condition depends on the choice of the penalty parameter $\lambda_N$ which has to be taken \enquote{sufficiently large} (but independent of $\ell$).
Together with an optimal approximation property of the space $\Vellg$ one obtains optimal order of convergence, i.e.
\begin{align} \label{eq:nitscheConvergence}
 \| u^* - u_h \|_{0,\Omega} \leq c h_\ell^2 \|u^*\|_{2,\BrokenOmega},
\end{align}
where $u^* \in H^2(\BrokenOmega)$  is the solution of problem \eqref{eq:ellipIf}. The constant $c$ in this error bound is independent of $\ell$ and the location of $\Gamma_\ell$ (but depends on $\mu$).
\subsection{Parameter-free CutFEM}\label{SectLehrenfeld}
As mentioned in the previous section, the choice of a sufficiently large  penalty parameter $\lambda_N$ is crucial for the coercivity of the bilinear form $a_h(\cdot,\cdot)$ used in the discretization \eqref{eq:nitscheBfa}. A drawback of the Nitsche method \eqref{eq:nitscheBfa} is that it requires this manually chosen stabilization parameter. Furthermore, if $\lambda_N$ is chosen globally this may lead  to overstabilization in some parts of the domain.
In   \cite{lehrenfeld2016removing}  a variant of the Nitsche method is presented that is parameter free.  We briefly describe this method and refer to  \cite{lehrenfeld2016removing} for more details and analysis of the method. 

We use the bilinear forms defined in \eqref{N1}-\eqref{N4}. The element bilinear forms, e.g. $a_{T}(\cdot,\cdot)$ corresponding to the element contribution to $a(\cdot,\cdot)$,  are defined in the usual way: $a(\cdot,\cdot) = \sum_{T \in \mathcal{T}_\ell} a_{T} (\cdot, \cdot)$.
On cut elements $T \in \Tellg$, $T_i = T \cap \Omega_{\ell,i}$,  a local lifting operator $\mathcal{L}_T : H^1(T_1 \cup T_2) \to \mathcal{P}^1_0(T_{1,2})$ is introduced, which maps into the space of polynomials of degree $1$ that are orthogonal to constants on each sub-element:
\begin{align*}
 \mathcal{P}_0^1(T_{1,2}) := \{ \,v \in L^2(T)~|~v|_{T_i} \in \mathcal{P}^1(T_i) / \mathcal{P}^0(T_i) \text{ for } i=1,2 \, \}. 
\end{align*}
The lifting operator is constructed such that it has the property
\begin{align}\label{eq:nitsche_paramfree_localprob}
 a_{T}( \mathcal{L}_T(u), v ) = N_T^c(v,u) \quad \text{for all}~~ v \in \mathcal{P}^1_0(T_{1,2}).
\end{align}
On uncut elements we set $\mathcal{L}_T(\cdot) = 0$ and thus obtain the global lifting operator
\begin{align*}
\mathcal{L} : H^1(\BrokenOmega) \to \bigoplus_{T \in \Tellg} \mathcal{P}_0^1(T_{1,2}) \quad \text{with} \quad \mathcal{L}(u)|_T := \mathcal{L}_T(u).
\end{align*}
The parameter free bilinear CutFEM, which we call P-Nitsche, is as follows: determine $u_h \in  \Vellg $ such that 
\begin{equation}\label{eq:nitsche_paramfree_bform} \tilde a_h( u_h, v_h ) = l(v_h):= (f,v_h)_\Omega \quad \text{for all}~~ v_h \in \Vellg, \tag{P-Nitsche}
\end{equation}
with 
\begin{align}\label{eq:nitsche_paramfree_bformA}
\tilde a_h(u,v) = a(u,v) + N^c(u,v) + N^c(v,u) + 2a( \mathcal{L}(u), \mathcal{L}(v) ) + N^s_1 (u,v).
\end{align} 
Through the lifting defined in \eqref{eq:nitsche_paramfree_localprob} the choice $\lambda_N=1$, as indicated by $N_1^s(\cdot,\cdot)$ in \eqref{eq:nitsche_paramfree_bformA}, suffices to guarantee coercivity. Sufficient additional  stability is implicitly added by the term $a(\mathcal{L}(\cdot), \mathcal{L}(\cdot))$. A proof of consistency, coercivity and continuity can be found in \cite{lehrenfeld2016removing}. These properties imply the same error bound as in  \eqref{eq:nitscheConvergence}. 
\newline

\begin{remark} \rm
In assembling the stiffness matrix corresponding to the bilinear form \eqref{eq:nitsche_paramfree_bformA}, for the evaluation of $\mathcal{L}_T$ one has  to solve the local systems in \eqref{eq:nitsche_paramfree_localprob}, which requires the inversion of element matrices. Further implementational aspects  are discussed  in \cite{lehrenfeld2015space,lehrenfeld2016removing}.
\end{remark}

\subsection{Coefficient-stable CutFEM}
The methods introduced in section \ref{HansboMeth} and \ref{SectLehrenfeld} are not robust with respect to large jumps in the diffusion coefficients, i.e. for $\mu_{\max} / \mu_{\min} \gg 1$. Hence, we also consider the method from \cite{burman2011numerical} which has better robustness properties w.r.t. the variation of the coefficient jump  $\mu_{\max} / \mu_{\min}$.  In that method the averaging operator is defined as
\begin{align*}
 \average{u_h}_\mu := ( \kappa_1^\mu u_{h,1} + \kappa_2^\mu u_{h,2} )|_{\Gamma_\ell}, \quad u_h \in \Vellg,
\end{align*}
with harmonic weights
\begin{align*}
\kappa_1^{\mu} = \frac{\mu_2}{\mu_1 + \mu_2}, \quad \kappa_2^\mu = \frac{\mu_1}{\mu_1 + \mu_2}, \quad \text{(hence, $\kappa_1^\mu + \kappa_2^\mu = 1$)}.
\end{align*}
We emphasize the harmonic averaging by the subscript $\mu$ in the averaging operator.  It turns out that using these weights  the method is not robust with respect to the interface location, i.e., there occur instabilities due to ``small cuts''.  To recover the stability, additional terms are included in the bilinear form, which are based on  the so-called ghost penalty stabilization \cite{Burman}:
\begin{align} \label{eq:ghostpen}
g_h(u_h,v_h) := \sum_{i=1}^2 \sum_{F \in \mathcal{F}^\ell_i} \varepsilon_g \, \mu_i \, h_F  
( \jump{\nabla u_{h,i} \cdot \vv{n}_F}_F, \jump{\nabla v_{h,i} \cdot \vv{n}_F}_F )_F
\end{align}
with a (constant) stabilization parameter $\varepsilon_g >0$.
Here, $\jump{\nabla v_{h,i} \cdot \vv{n}_F}_F$ denotes the jump of the normal component of the piecewise constant function $\nabla v_{h,i}$ across the face $F$. The sets of faces only contain faces close to the interface and are defined as
\begin{align*}
\mathcal{F}^\ell_i := \{ F \subset \partial T  \mid T \in \mathcal{T}_\ell^\Gamma, F \not\subset \partial \Omegaie \} \quad \text{for } i=1,2.
\end{align*}
An illustration for different triangulation levels is shown in Figure~\ref{fig:gpgrids}.
\vspace{0.5cm}
\begin{figure}[ht!]
\centering
\begin{subfigure}[c]{0.45\textwidth}
\centering
\begin{overpic}[width=0.8\textwidth]{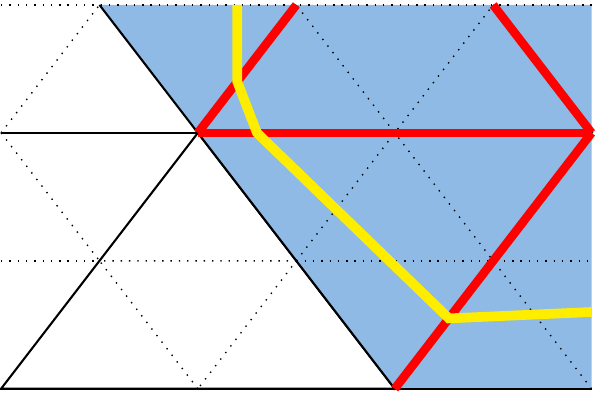}
  \put(105,10){\color{black}$\Gamma_\lmone$}
  \put(70,50){$\Omega_{\lmone,1}^e$}
    \put(50,70){\color{red}$\mathcal{F}_1^\lmone$}
\end{overpic}
\subcaption{Face set $\mathcal{F}_1^\lmone$ (in red) on $\mathcal{T}_{\ell-1}$}
\end{subfigure}
\begin{subfigure}[c]{0.45\textwidth}
\centering
\begin{overpic}[width=0.8\textwidth]{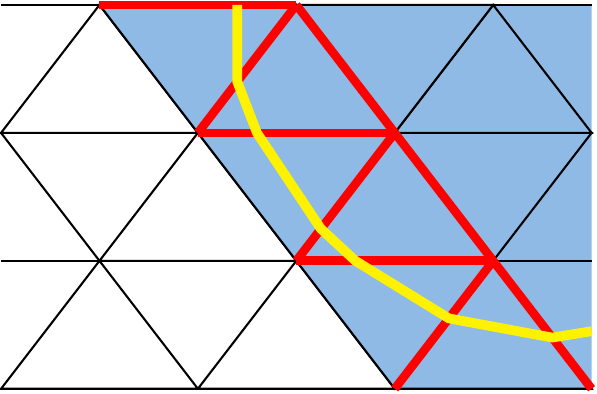}
  \put(105,10){\color{black}$\Gamma_\ell$}
  \put(80,50){$\Omega_{\ell,1}^e$}
  \put(50,70){\color{red}$\mathcal{F}_1^\ell$}
\end{overpic}
\subcaption{Face set $\mathcal{F}_1^\ell$ (in red) on $\mathcal{T}_{\ell}$}
\end{subfigure}
\caption{Face sets for ghost penalty stabilization on different refinement levels.}
\label{fig:gpgrids}
\end{figure}
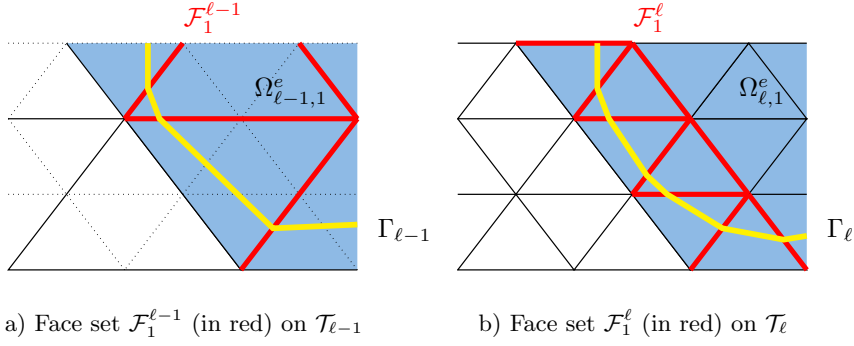

The resulting discretization method, which we call $\mu$-Nitsche, is as follows: determine $u_h \in \Vellg$ such that
\begin{align} \label{eq:munitsche}
\hat{a}_h(u_h, v_h) = l(v_h) := (f,v_h)_\Omega  \quad \text{for all}~~ v_h \in \Vellg, \tag{$\mu$-Nitsche}
\end{align}
with 
\begin{align*}
\hat{a}_h(u,v) &:= a(u,v) + N^{c,\mu}(u,v) + N^{c,\mu}(v,u) + N^s_{\lambda^\mu}(u,v) + g_h(u,v), \\
N^{c,\mu}(u,v) &:= ( \average{ - \mu \nabla u \cdot \vv{n}_\ell }_\mu, \jump{v} )_{\Gamma_\ell}, \\
 \lambda^\mu  & :=  \frac{2 \mu_1 \mu_2}{\mu_1 + \mu_2} \lambda_N.
\end{align*}
The penalty parameter $\lambda^\mu$ is $\mu$-dependent (with $\lambda_N > 0$  an $\ell$- and $\mu$-independent parameter value).
In \cite{burman2011numerical} the estimate $\vertiii{u^* - u_h}_{1,h,\Omega} \leq c h \|u^*\|_{H^2(\BrokenOmega)}$ was proved with norm definition
\begin{align*}
\vertiii{v_h}^2_{1,h,\Omega} := \sum_{i=1}^2 \| \mu_i^\frac12 \nabla v_{h,i} \|_{0,\Omega}^2
+ \| \average{\mu}_\mu^\frac12 \average{ \nabla v_h \cdot \vv{n}_\Gamma} \|_{-\frac12,h,\Gamma}^2 
+ \| \average{\mu}_\mu^\frac12 \jump{v_h} \|^2_{\frac12,h,\Gamma},
\end{align*}
where $\| v_h \|_{\pm\frac12,h,\Gamma}^2 := \sum_{T \in \mathcal{T}_\ell^\Gamma}   h_T^{\mp 1} \| v_h \|_{0,\Gamma_T}^2$.
This error bound can be  combined with a duality argument, resulting in  an optimal error estimate as in  \eqref{eq:nitscheConvergence}. 

\section{Multigrid method} \label{sectMG}
In this section we present an efficient multigrid solver for the discrete problems resulting from \eqref{eq:nitscheBfa}, \eqref{eq:nitsche_paramfree_bform} or \eqref{eq:munitsche}. 
These discretizations use the same unfitted finite element space $\Vellg $.

For a multigrid solver we have to specify how restrictions and prolongations are constructed and which smoother is used. We first address the restrictions. 
Given an appropriate prolongation (explained below) $\bp_\ell$ (in matrix representation), we make the canonical choice for the restriction, namely $\br_\ell:=\bp_\ell^T$. The coarse grid matrices are constructed by the usual procedure in geometric multigrid methods, namely \emph{direct discretization on the coarser levels}. Alternatively one could consider the construction of coarse grid matrices by the Galerkin approach ($A_\lmone = \vv{r}_\ell A_\ell \vv{p}_\ell$). It turns out that both methods yield similar results; cf. Remark~\ref{RemGalerkinapproach}.  In the following two subsections we discuss the prolongation operator and the smoother. 
\subsection{Prolongation operator} \label{sec:prol}
For the derivation and representation of the prolongation operator we use the basis \eqref{basis}.
Recall that the unfitted finite element spaces are not necessarily nested. We introduce a natural prolongation $p_\ell: \, \Vellmg \to \Vellg$. Note that due to \eqref{ass2} the spaces $V_{\ell,i}$ are nested in the sense that  for $w \in  V_{\ell-1,i} $ we have ${w}_{|\Omega_{\ell,i}^e}\in V_{\ell,i}$, $i=1,2$. We introduce the operator
\begin{align}\label{eq:levelcutoff}
 \Iellg : \Vellmi{1} \times \Vellmi{2} \to \Velli{1} \times \Velli{2}, \quad \Iellg \begin{pmatrix} u_{H,1} \\ u_{H,2} \end{pmatrix} := \begin{pmatrix} \Iellgi{1} u_{H,1} \\ \Iellgi{2} u_{H,2}  \end{pmatrix}
\end{align}
with 
\begin{align}
\Iellgi{i}: \Vellmi{i} \to \Velli{i}, \quad \Iellgi{i} u_{H,i}  := u_{H,i}|_{\Omegaie},
\end{align}
the natural embedding (by restriction) of a coarse level function $u_{H,i}$ (defined on $\Omega_{\ell-1,i}^e$) into $\Velli{i}$. 
 Hence for $u_H = \cR_{\ell-1} (u_{H,1}, u_{H,2})^T \in \Vellmg $ 
the prolongation
\begin{equation} \label{defprol}
 p_\ell u_H := \cR_{\ell} \Iellg \cR_{\ell-1}^{-1} u_H = \cR_{\ell} \big( {u_{H,1}}_{|\Omega^e_{\ell,1}}, {u_{H,2}}_{|\Omega^e_{\ell,2}} \big)^T \in \Vellg   
\end{equation}
is well-defined. The prolongation procedure is illustrated in Figure~ \ref{fig:twodominterpolation}.
\begin{figure}[ht!]
\centering
\begin{subfigure}[c]{0.45\textwidth}
\centering
\input{twodom1.tex}
\subcaption{$\cR_\lmone^\mone u_H$ on $\mathcal{T}_{\ell-1}$}
\end{subfigure}
\begin{subfigure}[c]{0.45\textwidth}
\centering
\input{twodom2.tex}
\subcaption{$\cR_\ell^\mone p_\ell u_H$ on $\mathcal{T}_{\ell}$}
\end{subfigure}
\caption{Illustration of prolongation procedure.}
\label{fig:twodominterpolation}
\end{figure}
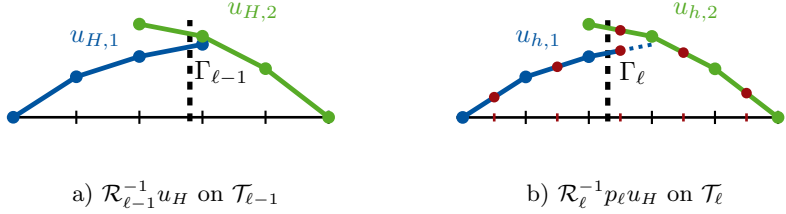

Using the results in \eqref{eq:feisomorph} we derive the  matrix representation of this prolongation operator. For this we use the matrix representation of the prolongation corresponding to $V_{\ell,i}$, i.e. 
\begin{align*}
\bp_{\ell,i}:= P_{\ell,i}^{-1} \Iellgi{i} P_{\ell-1,i}.
\end{align*}
 Note that for this operator to be well-defined we need \eqref{ass2}, because for $\bu \in \R^{n_{\ell-1,i}}$ we have that $w:=P_{\ell-1,i}\bu$ is defined (only) on  $\Omega_{\ell-1,i}^e$; due to \eqref{ass2} we have that  $P_{\ell,i}^{-1} \Iellg P_{\ell-1,i}\bu =P_{\ell,i}^{-1} w_{|\Omega_{\ell,i}^e}$ is well-defined. 
\begin{lemma} \label{lemprol}
For the prolongation in \eqref{defprol} the following holds:
 \begin{equation}
     p_{\ell} P_{\ell-1}^\Gamma   \begin{pmatrix}  \bu_{H,1} \\ \bu_{H,2} \end{pmatrix} = P_\ell^\Gamma \begin{pmatrix} \bp_{\ell,1}\bu_{H,1} \\  \bp_{\ell,2}\bu_{H,2}         \end{pmatrix}      \quad \text{for}~~\begin{pmatrix}  \bu_{H,1} \\ \bu_{H,2} \end{pmatrix} \in \R^{n_{\ell-1,1}+n_{\ell-1,2}}.
 \end{equation}
\end{lemma}
\begin{proof}
 Take $\begin{pmatrix}  \bu_{H,1} \\ \bu_{H,2} \end{pmatrix} \in \R^{n_{\ell-1,1}+n_{\ell-1,2}}$ 
 and
 \[
    u_H:= P_{\ell-1}^\Gamma   \begin{pmatrix}  \bu_{H,1} \\ \bu_{H,2} \end{pmatrix} = \cR_{\ell-1}
    \begin{pmatrix}
     P_{\ell-1,1}  \bu_{H,1} \\ P_{\ell-1,2} \bu_{H,2}
    \end{pmatrix},
 \]
where in the last equality we used \eqref{eq:feisomorpha}. Hence, from \eqref{defprol} and  \eqref{eq:feisomorphb} we get
\[ \begin{split}
  p_{\ell} P_{\ell-1}^\Gamma   \begin{pmatrix}  \bu_{H,1} \\ \bu_{H,2} \end{pmatrix} & = p_{\ell} u_H = \cR_\ell \Iellg  \begin{pmatrix}
     P_{\ell-1,1}  \bu_{H,1} \\ P_{\ell-1,2} \bu_{H,2}
    \end{pmatrix}  \\ & =  P_\ell^\Gamma \begin{pmatrix}
     P_{\ell,1}^{-1} \Iellgi{1} P_{\ell-1,1}  \bu_{H,1} \\ P_{\ell,2}^{-1} \Iellgi{2} P_{\ell-1,2} \bu_{H,2}
    \end{pmatrix}= P_\ell^\Gamma \begin{pmatrix} \bp_{\ell,1}\bu_{H,1} \\  \bp_{\ell,2}\bu_{H,2}         \end{pmatrix}, \end{split}
\]
which proves the result.
\end{proof}
\ \\
As a direct corollary we obtain that the matrix representation of the prolongation in \eqref{defprol}, i.e., $ \bp_\ell= (P_{\ell}^\Gamma)^{-1}p_\ell P_{\ell-1}^\Gamma$,  is given by
\begin{equation} \label{matprol}
 \bp_\ell \begin{pmatrix}  \bu_{H,1} \\ \bu_{H,2} \end{pmatrix} = \begin{pmatrix} \bp_{\ell,1}\bu_{H,1} \\  \bp_{\ell,2}\bu_{H,2}         \end{pmatrix}      \quad \text{for}~~\begin{pmatrix}  \bu_{H,1} \\ \bu_{H,2} \end{pmatrix} \in \R^{n_{\ell-1,1}+n_{\ell-1,2}}.
 \end{equation}
This is the prolongation used in the multigrid solver. 
It has the following interesting (and at first sight maybe surprising) property. The prolongation in matrix representation only requires  $ \bp_{\ell,i}\bu_{H,i}$, $i=1,2$, where $\bp_{\ell,i}$ is the prolongation corresponding to $V_{\ell,i}$. This $\bp_{\ell,i}$ does \emph{not} depend on $\Gamma_{\ell-1}$ or $\Gamma_\ell$ and is just the simple interpolation used in a standard setting of nested finite element spaces, without any cutting procedures involved.  For the matrix-vector representation of the discrete problems we use another basis as the one in \eqref{basis}, which  will be introduced in the next section. The matrix representation of the prolongation (in the basis \eqref{basis}) can easily be transformed to this other basis.

\subsection{Smoothers} \label{sec:smooth} 
Before we explain the smoother we first introduce the stiffness matrix corresponding to the discretization methods described in section~\ref{sec:disc}. In XFEM one usually chooses a basis of the space $\Vellg$ that differs from the one in \eqref{basis}. For the matrix-vector representation of the unfitted finite element discretization we also use this alternative basis, which we now introduce. Recall that $(\varphi_{\ell,i}^k)_{1 \leq k \leq n_{\ell,i}}$  denotes the standard nodal basis in $V_{\ell,i}$; cf.~\eqref{deflocVh}. The vertex corresponding to $\varphi_{\ell,i}^k$ is denoted by $x_{k,i}$.   Let $\cV_i$ be the set of those vertices in $\Omega_{\ell,i}$ that correspond to exactly two degrees of freedom in $\Vellg$ (namely in $V_{\ell,1}$ and $V_{\ell,2}$). We define:
\begin{equation} \label{defB}
 B_i:=\{\, \varphi_{\ell,i}^k~|~x_{k,i} \in {\Omega}_{\ell,i} \,\},~~B_i^\Gamma:= \{\,\varphi_{\ell,i}^k|_{\Omega \setminus \Omega_{\ell,i}} \mid x_{k,i} \in \cV_i\,\},~~i=1,2. 
\end{equation}
Note that $B_1 \cup B_2$ is the nodal basis in the standard finite element space $V_\ell$ \eqref{basis}. The disjoint union $B_1 \cup B_2 \cup B_1^\Gamma \cup B_2^\Gamma$ is a basis of $\Vellg$, which we call XFEM basis. The corresponding finite element isomorphism is denoted by $\tilde P_{\ell}^\Gamma:\, \Bbb{R}^{n_{\ell,1}+n_{\ell,2}} \to V_\ell^\Gamma$. It is easy to transform between this XFEM basis and the one given in \eqref{basis}. For the matrix-vector  representation of the discretizations we use the XFEM basis. 
Hence, the stiffness matrix $\bA_\ell \in \R^{(n_{\ell,1}+n_{\ell,2})\times(n_{\ell,1}+n_{\ell,2}) }$ for \eqref{eq:nitscheBfa} is defined by (with $\langle \cdot, \cdot \rangle$ the Euclidean scalar product)
\begin{equation} \label{stiffnessmatrix}
 \langle \bA_\ell \bu, \bv \rangle = a_h(\tilde P_\ell^\Gamma \bu, \tilde P_\ell^\Gamma \bv), \quad \bu, \bv \in \R^{n_{\ell,1}+n_{\ell,2}},
\end{equation}
and similarly for  \eqref{eq:nitsche_paramfree_bform} and \eqref{eq:munitsche}.

\begin{remark} \rm
 Clearly there is an issue related to ``small cuts''. In \cite{Reusken07} for the case of the XFEM basis  it is shown that the poor conditioning of this basis with respect to the $L^2$-scalar product can be eliminated by a simple scaling. More precisely, let $\bM_\ell$ be the mass matrix formed by the $L^2(\Omega)$ scalar product of the XFEM basis functions and $\bD_\ell:={\rm diag}(\bM_\ell)$. Then $\bD^{-1}_\ell\bM_\ell$ has a \emph{uniformly bounded spectral condition number}, where the uniformity holds both with respect to the refinement level $\ell$ and the way in which the triangulation is intersected by the (approximate) interface $\Gamma_\ell$.     Such a robustness property w.r.t. the location of $\Gamma_\ell$ does in general \emph{not} hold for the diagonally scaled stiffness matrix. In particular 
in the discretizations  \eqref{eq:nitscheBfa} and \eqref{eq:nitsche_paramfree_bform} there are no (e.g. ghost penalty type) stabilizations that yield robustness of the conditioning of the stiffness matrix  $\bA_{\ell}$ w.r.t. the location of the interface. This effect is illustrated in Remark~\ref{Remcondnumber}. 
 Despite  this possibly very poor conditioning, the multigrid solver introduced below has  a contraction number that is small uniformly in the level number $\ell$, the 
interface location and (for the method \eqref{eq:munitsche}) the jump in  the diffusion coefficient. 
\end{remark}

Concerning the choice of the smoother we distinguish two problem classes. If the quotient of the largest and smallest  diffusion coefficient  is moderate, i.e.,  $\mu_{\max}/\mu_{\min} =\mathcal{O}(1)$, we use a \emph{standard Gauss-Seidel smoother}. From the experiments we see that this yields satisfactory results.
If, however, this quotient is increased we observe a deterioration of the rate of convergence. 

For the case of large coefficient jumps $\mu_{\max}/\mu_{\min} \gg 1$ we propose a modification in which 
an additional local error damping procedure  is included that acts \emph{only close to the interface}. One obvious possibility, used also in multigrid methods in other settings, is to apply local smoothing (cf.\cite{bramble1993multigrid,griebel1995abstract,janssen2011adaptive}), in our case that would correspond to adding extra Gauss-Seidel smoothing iterations at grid points close to the interface. From experiments we see that this then results in a fairly efficient multigrid solver, also for larger diffusion quotient ratios.  We propose, however, another local correction procedure which results in a solver with (even) better performance.
This smoother, which we call ``Gauss-Seidel with interface correction'' (GS-IC), is based on a simple domain decomposition approach and is explained below. 

\subsubsection*{Gauss-Seidel with interface correction} 
The smoother (on level $\ell$) needs only information from the discretization on level $\ell$. To simplify notation, in the remainder of this section we delete the level index $\ell$. The set of those vertices in $\Omega_i$ that correspond to exactly two degrees of freedom is denoted by $\cV_i=\{x_{1,i}, \ldots, x_{m_i,i}\}$. Let
$\mathcal{I}^\Gamma \subset \{ 1,\ldots,n_1+n_2\}$ be the subset which contains the (global) numbering of the unknowns corresponding to the vertices in $\cup_{i=1,2} \cV_i$ and $R^\Gamma: \, \Bbb{R}^{n_1+n_2} \to \Bbb{R}^{|\cIG|}$ the matrix representation of the corresponding injection. Note that $|\cIG|=2(m_1+m_2)$. 
With that we  define the \emph{local interface matrix}
\begin{equation} \label{Aloc}
A^\Gamma = R^\Gamma A \, (R^\Gamma)^T\in \Bbb{R}^{|\cIG|\times |\cIG|},
\end{equation}
which is the part of the full stiffness matrix $A$ corresponding to the degrees of freedom  ``close to'' the interface.
We propose a local interface correction smoother in the spirit of a multiplicative Schwarz smoother, where first a standard smoothing, e.g. Gauss-Seidel, is performed on the entire domain followed by a correction for the interface unknowns.

 \begin{algorithm}[H]
      \setstretch{1.5}
      \caption{Gauss-Seidel with interface correction (GS-IC)}
      \label{alg:ifAsSmooth}
      \begin{algorithmic}[1] 
      \Function{$S^C_\Gamma$}{$\vv{x}^n,\vv{b},A$}
      \State $\vv{x}^{n+\frac12} = S(\vv{x}^n,\vv{b},A)$ \Comment{perform standard smoothing}
      \State $\vv{r} = A \vv{x}^{n+\frac12}  - \vv{b}$ \Comment{new residual}
      \State  $\vv{x}^{n+1} = \vv{x}^{n+\frac12} - (R^\Gamma)^T (A^\Gamma)^\mone \, R^\Gamma \vv{r}$ \Comment{interface update}
       \State \Return $\vv{x}^{n+1}$ 
       \EndFunction     
       \end{algorithmic}            
     \end{algorithm}
     
Systems with the matrix $A^\Gamma$ can be solved (approximately) with arithmetic costs that are (very) low compared to the costs of the Gauss-Seidel iteration in step 2. This will be further addressed in 
section~\ref{sectlocalA}.

\subsection{Analysis of local interface matrix $A^\Gamma$} \label{sectlocalA}
In the case of moderate or large coefficient jumps we propose to use the GS-IC smoother. For moderate coefficient jumps the discretization methods \eqref{eq:nitscheBfa} and  \eqref{eq:nitsche_paramfree_bform} yield satisfactory results. Below, in section~\ref{sectcond} it is shown that for these methods the spectral condition number of the diagonally preconditioned matrix $A^\Gamma$ is bounded \emph{independent of the level $\ell$ and of the location of $\Gamma$ in the triangulation}. From numerical experiments (cf. section~\ref{sec:exp}) we see that a fixed (independent of $\ell$) low number of (preconditioned) conjugate gradient (CG) iterations is sufficient for the \emph{approximate} solution of the linear system in step 4 of GS-IC. Since the number of unknowns in the $A^\Gamma$-system is much lower than the total number of unknowns, $|\cIG| \ll n_1+n_2$, and $A^\Gamma$ has the same sparsity pattern as $A$, we thus obtain a very efficient realization of the GS-IC smoother. The condition number  of the diagonally preconditioned matrix $A^\Gamma$ depends on the size of the coefficient jump. For very large jumps the condition numbers 
are very large and one needs (too) many CG iterations. In such a case the discretization method \eqref{eq:munitsche} is significantly better than the other two. For the  linear system resulting from this method we propose to use the multigrid solver with GS-IC smoother and in step 4 we apply a  \emph{sparse direct solver}. This still is an efficient smoother as further explained in section~\ref{sectdirect}.

\subsubsection{Conditioning of matrix $A^\Gamma$} \label{sectcond}
For an analysis of conditioning properties of the matrix $A^\Gamma \in \Bbb{R}^{|\cIG|\times |\cIG|}$ it is convenient to introduce the corresponding representation on the finite element space $V^\Gamma$ (we delete the level index $\ell$). To each $x_{k,i} \in \cV_i$  there corresponds one standard nodal finite element function and one cut finite element basis function. These are denoted by  $\varphi_{i}^k$ and $\varphi_{i,\Gamma}^k:= (\varphi_{i}^k)_{|(\Omega \setminus \Omega_{i})}$, respectively ($1 \leq k \leq m_i$).   We introduce the (local) subspaces of $V^\Gamma$:
\begin{align}
 \cW_0 & := {\rm span}\{\, \varphi_{i}^k~|~k=1, \ldots, m_i,~~i=1,2\,\} \nonumber \\
 V_i^\text{ex} &:= {\rm span}\{\, \varphi_{i,\Gamma}^k~|~k=1, \ldots, m_i\,\}, ~i=1,2, ~~\cW_1:=V_1^\text{ex} \oplus V_2^\text{ex} \label{Basis} \\
 V_\text{loc}& :=\cW_0\oplus \cW_1 \subset V^\Gamma.\nonumber
\end{align}
For the finite element isomorphism w.r.t. the XFEM basis $\tilde P^\Gamma:\, \Bbb{R}^{n_1+n_2} \to V^\Gamma$ we have, per construction, that $ \hat P^\Gamma:\,\Bbb{R}^{|\cIG|} \to V_\text{loc}$ is an isomorphism and the local matrix $A^\Gamma$ in \eqref{Aloc} is the matrix representation of $a_h(\cdot,\cdot)$ restricted to $V_\text{loc}$:
\[
  \la A^\Gamma \bu,\bv\ra = a_h(\hat P^\Gamma \bu,\hat P^\Gamma \bv), \quad \text{for all}~~\bu, \bv \in \Bbb{R}^{|\cIG|}.
\]
For the analysis of the spectral condition number of $A^\Gamma$ we will use the key property (explained in more detail below) that the splitting $V_\text{loc}=\cW_0\oplus \cW_1$ is a \emph{stable splitting} w.r.t. $a_h(\cdot,\cdot)$. The restrictions of $  \hat P^\Gamma$ to  the subspace $\cW_s$ are denoted by $\hat P_s$, $s=0,1$:
\begin{align*}
  \hat P_0: & \, \Bbb{R}^{m_1+m_2} \to \cW_0, \quad \hat P_0 \hat \bu := \sum_{i=1}^2\sum_{k=1}^{m_i} \hat u_{i}^k \varphi^k_{i}, \\
  \hat P_1: & \, \Bbb{R}^{m_1+m_2} \to \cW_1, \quad \hat P_1 \tilde \bu := \sum_{i=1}^2\sum_{k=1}^{m_i} \tilde u_{i}^k \varphi_{i,\Gamma}^k.
\end{align*}
The unknowns in $\cIG$ are numbered accordingly, taking first the unknowns corresponding to $\cW_0$ and then those corresponding to $\cW_1$. This is denoted by the partitioning $ \Bbb{R}^{|\cIG|}\ni \bu =\begin{pmatrix} \hat \bu \\ \tilde \bu \end{pmatrix}$. We write $V_\text{loc} \ni u = \hat P^\Gamma \bu = \hat P_0 \hat \bu +\hat P_1 \tilde \bu =: \hat u + \tilde u \in \cW_0+\cW_1$ for the splitting of a finite element function. Also the matrix $D:= {\rm diag}(A^\Gamma)$ is partitioned accordingly, i.e., 
\[ D= \begin{pmatrix} D_0 & \emptyset \\ \emptyset & D_1 \end{pmatrix}, \quad  \text{and} ~~
  \la D\bu, \bu\ra = \la D_0 \hat \bu, \hat \bu \ra + \la D_1 \tilde \bu, \tilde \bu \ra.
\]
We introduce a representation of the Jacobi preconditioner $D_s$ in the subspace $\cW_s$, i.e., $B_s: \, \cW_s \to \cW_s$ such that
\[
  (B_s w,z)_{L^2}:= \la D_s \hat P_s^{-1} w, \hat P_s^{-1} z \ra \quad \text{for all}~~w,z \in \cW_s,~~s=0,1.
\]
For arbitrary $\bu \in \Bbb{R}^{|\cIG|}$, $\bu \neq 0$, we then have
\[
 \frac{\la A^\Gamma \bu,\bu \ra}{\la D \bu,\bu\ra} = \frac{ a_h(u,u)}{(B_0 \hat u,\hat u)_{L^2} +(B_1 \tilde u,\tilde u)_{L^2}}. 
\]
We now first consider the standard Nitsche bilinear form $a_h(\cdot,\cdot)$ used in \eqref{eq:nitscheBfa}. In \cite{lehrenfeld2017optimal}, for the two-dimensional case $d=2$ the following \emph{stable splitting property} is derived: there exists a  constant $K >0$, independent of the discretization $\ell$ and of how the triangulation is intersected by $\Gamma$ such that
\begin{equation} \label{stable}
 K \big( a_h(\hat u,\hat u)+a_h(\tilde u,\tilde u) \big) \leq a_h(u,u) \leq 2\big(a_h(\hat u,\hat u)+a_h(\tilde u,\tilde u)\big) \quad \forall~u=\hat u+ \tilde u \in V_\text{loc}.
\end{equation}
From numerical experiments (given in \cite{lehrenfeld2017optimal}) and further theoretical results it is expected that this result is also valid for $d=3$. We thus obtain
\[
  \frac{\la A^\Gamma \bu,\bu \ra}{\la D \bu,\bu\ra} \sim \frac{a_h(\hat u,\hat u)+a_h(\tilde u,\tilde u) }{(B_0 \hat u,\hat u)_{L^2} +(B_1 \tilde u,\tilde u)_{L^2}},
\]
where the $\sim$ is used to denote that there are inequalities between the terms on both sides with strictly positive constants independent of the discretization $\ell$ and of how the triangulation is intersected by $\Gamma$. It is easy to show (cf. Lemma 4.2 in \cite{lehrenfeld2017optimal}) that the Jacobi preconditioner and $a_h(\cdot,\cdot)$ are spectrally equivalent on $\cW_1$ (in fact even on the two subspaces $V_i^{ex}$ of $\cW_1$):
\begin{equation} \label{spec1}
  \frac{a_h(\tilde u,\tilde u) }{B_1 (\tilde u,\tilde u)_{L^2}} \sim 1 \quad \text{for all}~~\tilde u \in \cW_1.
\end{equation}
The proof of this result in \cite{lehrenfeld2017optimal} holds for both $d=2$ and $d=3$. It remains to derive a similar spectral equivalence on $\cW_0$. Such a result is given in the following lemma.
\begin{lemma} \label{lemspectral} The spectral equivalence 
 \begin{equation} \label{specAA}
  \frac{a_h(\hat u,\hat u) }{B_0 (\hat u,\hat u)_{L^2}} \sim 1 \quad \text{for all}~~\hat u \in \cW_0
\end{equation}
holds.
\end{lemma}
\begin{proof}
Take $w \in \cW_0$. We write $w= \sum_{i=1}^2 \sum_{k=1}^{m_i} w_{i}^k \varphi_{i}^k=: \sum_{i,k} w_{i}^k \varphi_{i}^k$, and $\hat P_0^{-1} w= \bw = (w_{i}^k)_{1\leq i \leq 2, 1 \leq k \leq m_i}$. Because $w$ is continuous we have $a_h(w,w)=a(w,w)= \int_\Omega \mu (\nabla w)^2 \, dx$. Due to the finite overlap property of finite element nodal functions we have
\[
 a_h(w,w) \lesssim \sum_{i,k} (w_{i}^k)^2 \int_\Omega \mu \big(\nabla \varphi_{i}^k\big)^2 \, dx = \la D_0 \bw,\bw\ra
\]
and thus
\begin{equation} \label{upperb}
 a_h(w,w) \lesssim (B_0 w,w)_{L^2}
\end{equation}
holds, which yields the ``$\lesssim$'' inequality in \eqref{specAA}. The constant in $\lesssim$ depends only on shape regularity properties of the triangulation. For every vertex $x_{k,i} \in \cV_i$, $i=1,2$, $k=1,\ldots,m_i$, we can select a simplex $T_{k,i} \in \T$ such that: 1. $x_{k,i}$ is a vertex of $T_{k,i}$, 2. at least one vertex of $T_{k,i}$ is \emph{not} in $\cV_1 \cup \cV_2$, 3. $T_{k,i} \cap T_{m,j}= \emptyset$ if $(k,i) \neq (m,j)$. See Figure \ref{fig:nodesets} for illustration.
\begin{figure}[h]
\centering
\includegraphics[width=0.4\textwidth]{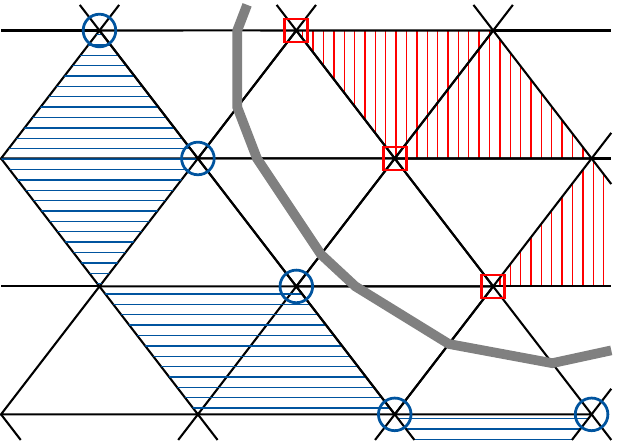}
\caption{Node sets $\mathcal{V}_1$ {\protect\tikz \protect\draw[color=rwthBlue] plot[mark=o, mark options={scale=2}] (0,0);} and $\mathcal{V}_2$ {\protect\tikz \protect\draw[color=red] plot[mark=square, mark options={scale=2}] (0,0);} with corresponding simplices $T_{k,1}$ {\protect\tikz \protect\draw[pattern=horizontal lines, pattern color=rwthBlue] (0,0) rectangle (0.3,0.3);} and $T_{k,2}$ {\protect\tikz \protect\draw[pattern=vertical lines, pattern color=red] (0,0) rectangle (0.3,0.3);}.}
\label{fig:nodesets}
\end{figure}
 Below we use (different) strictly positive constants, denoted by $c$, which depend only on the shape regularity of the triangulation. Note that for $w \in \cW_0$ the local linear function $w_{|T_{k,i}}$ has a zero value at at least one vertex of $T_{k,i}$, and thus
\[
  \nabla w_{|T_{k,i}} \cdot  \nabla w_{|T_{k,i}} \geq c \big( \frac{w(x_{k,i})}{h}\Big)^2.
\]
Hence,
\begin{equation} \label{ppl}
 \begin{split}
 a_h(w,w) & = \int_\Omega \mu \nabla w \cdot \nabla w \,dx \geq \mu_{\min} \sum_{i,k} \int_{T_{k,i}} \nabla w \cdot \nabla w \, dx  \\ & \geq \mu_{\min} c h^{-2} \sum_{i,k} |T_{i,k}|w(x_{k,i})^2 
   \\ & \geq  \mu_{\min}c h^{d-2} \sum_{i,k} w(x_{k,i})^2 = \mu_{\min} c h^{d-2} \la \bw,\bw \ra. 
\end{split}
\end{equation}
For the diagonal matrix $D_0$ we have 
\[
  (D_0)_{(k,i),(k,i)}= a_h(\varphi_{i}^k,\varphi_{i}^k) \leq \mu_{\max} \int_\Omega (\nabla \varphi_{i}^k)^2 \, dx \leq  \mu_{\max} c h^{d-2},
\]
and thus
\[
  (B_0 w,w)_{L^2}= \la D_0 \bw,\bw \ra \leq \mu_{\max} c h^{d-2} \la \bw,\bw \ra.
\]
Combining this with \eqref{ppl} we obtain the ``$\gtrsim$'' estimate in \eqref{specAA}.
\end{proof}
\ \\[1ex]
From these results it follows that for the standard Nitsche method the resulting local interface matrix $A^\Gamma$ has the property
\begin{equation}
 \frac{\la A^\Gamma \bu,\bu \ra}{\la D \bu,\bu\ra} \sim 1 \quad \text{for all}~~\bu\in \Bbb{R}^{|\cIG|},~\bu \neq 0.
\end{equation}
Hence, the diagonally preconditioned local interface matrix $D^{-1}A^\Gamma$ has a uniformly (w.r.t. discretization level and interface position) bounded condition number. A standard conjugate gradient method can be used as an efficient approximate solver for the system occurring in step 4 of the GS-IC preconditioner. 

We now briefly address the  parameter-free CutFEM, with bilinear form $\tilde a_h(\cdot,\cdot)$ as in  \eqref{eq:nitsche_paramfree_bform}. In \cite{lehrenfeld2016removing} it is shown that $\tilde a_h(\cdot,\cdot) \sim a_h(\cdot,\cdot)$ on the unfitted finite element space $V_\ell^\Gamma$ holds. Hence, the conditioning results derived above for the local interface matrix $A^\Gamma$ resulting from $a_h(\cdot,\cdot)$ (i.e., standard Nitsche) also hold for $\tilde a_h(\cdot,\cdot)$ (parameter-free Nitsche). 

These results on the conditioning of $A^\Gamma$ are illustrated in section~\ref{sectlarge}; cf. Table 
\ref{tab:mg_exsolve} and Table \ref{tab:condnum_agamma}.

\subsubsection{Direct sparse solver for  $A^\Gamma$} \label{sectdirect}
We discuss the computational complexity of solving $A^\Gamma \vv{x} = R^\Gamma \vv{r}$ in step 4 via a sparse direct solver. We consider $d=3$ and assume that $\Gamma$ is simply connected (as in the examples in section~\ref{sec:exp}). The  number of unknowns  in the system $A_\ell \bu = \mathbf{b}_\ell$ is of order $\mathcal{O}(h_\ell^{-3})$ and the computational costs of one Gauss-Seidel smoothing iteration is also  $\mathcal{O}(h_\ell^{-3})$. The local interface matrix $A^\Gamma$ has dimension $|\cIG| \times |\cIG|$, with $
|\cIG| =\mathcal{O}(h_\ell^{-2})$. Assume that the vertices $x_{k,i}$, $i=1,2$, $k=1,\ldots, m_i$, are ordered with a breadth first search traversal of the corresponding triangulation (starting from an arbitrary vertex). The standard finite element mass matrix (with this ordering), i.e., $M \in \Bbb{R}^{\frac12|\cIG|\times\frac12|\cIG|}$ with entries $M_{(k,i),(m,j)}= \int_\Omega \varphi_{i}^k \varphi_{j}^m\, dx$ has a bandwidth of size $\mathcal{O}(h_\ell^{-1})$. In the system $ A^\Gamma \vv{x} = R^\Gamma \vv{r}$ we use a blocking of the unknowns with at each vertex $x_{k,i}$ the two unknowns corresponding to $\varphi_{i}^k$ and $\varphi_{i,\Gamma}^k$; cf. \eqref{Basis}. Hence, the matrix $A^\Gamma$ has a block-structure with
\[
  (A^\Gamma)_{(k,i),(m,j)}= \begin{pmatrix} a_h(\varphi_{j}^m,\varphi_{i}^k) & a_h(\varphi_{j,\Gamma}^m,\varphi_{i}^k) \\
                              a_h(\varphi_{j}^m,\varphi_{i,\Gamma}^k) & a_h(\varphi_{j,\Gamma}^m,\varphi_{i,\Gamma}^k)
                            \end{pmatrix}.
\]
Since for all three discretization methods  $M_{(k,i),(m,j)}=0$ (disjoint supports of the basis functions) implies  $(A^\Gamma)_{(k,i),(m,j)}=0$, it follows that $A^\Gamma$ has a bandwidth of size $\mathcal{O}(h_\ell^{-1})$. Thus the computation of a sparse Cholesky factorization of $A^\Gamma$ requires $\mathcal{O}(h_\ell^{-3})$ arithmetic operations, which is of the same order of magnitude as the costs for one Gauss-Seidel iteration.
A numerical illustration of this optimal computational complexity for the sparse direct solver is given in section~\ref{sectlarge}, cf. Table \ref{tab:fill_in}.

\section{Numerical experiments} \label{sec:exp}
In the numerical experiments we use the domain $\Omega = [0,2]^3$ and two different types of interfaces $\Gamma$ (cf. Figure \ref{fig:domains}): 
\begin{itemize}
\item planar interface: 
$
\Gamma_\text{plan}(x_\Gamma) := \{ (x, y, z)^T \in \Omega \mid x = x_\Gamma \}, 
$

\item spherical interface. For given $\vv{m} \in \Bbb{R}^3$, $r >0$, we define the level set function $\phi(\vv{x}):= \| \vv{x} - \vv{m} \|_2^2 - r^2$ and 
$
\Gamma_\text{spher} (r,\vv{m}) := \left\{  \vv{x} \in \Omega \mid \phi(\vv{x})=0\, \right\}. 
$ Unless stated otherwise, we use $\vv{m}= (1.03,1.02,1.01)$, $r=0.413$.
\end{itemize}

The initial mesh $\mathcal{T}_0$ is quasi-uniform with $6\cdot 4^3$ tetrahedra and $\mathcal{T}_\ell$ for $\ell>1$ is obtained by uniform refinement of $\mathcal{T}_\lmone$, i.e. $| \mathcal{T}_\ell | = 8 \cdot |\mathcal{T}_\lmone|$. For the spherical interface the interface approximation in the experiments is given by $\Gamma_\ell = \{ \vv{x} \in \Omega \mid (I_{\ell+1}\phi)(\vv{x}) = 0\}$, with $I_{\ell+1}$ the nodal interpolation in the standard finite element space $V_{\ell+1}$. This iso$P_2$ choice (instead of the linear interpolation on the current mesh $\Tell$) is motivated by the fact that the level set function $\phi$ is a $P_2$ finite element function. We obtain very similar results if we use $ I_{\ell}\phi$ instead of $ I_{\ell+1}\phi$.
This interface approximation satisfies \eqref{ass1}, \eqref{ass2} in Assumption~\ref{ass}.

In all experiments the stiffness matrix corresponds to the XFEM basis \eqref{defB} and the  prolongation matrix \eqref{matprol} is transformed to this basis. 
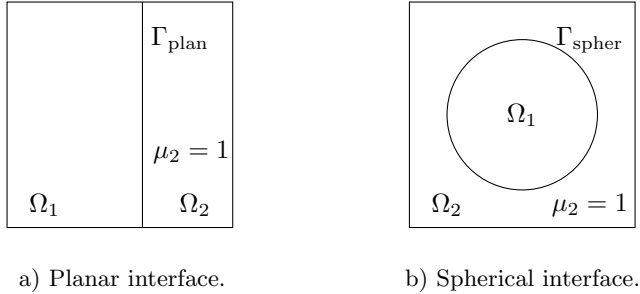
\begin{figure}[h!]
\centering
\begin{subfigure}[c]{0.4\textwidth}
\begin{center}
\begin{tikzpicture}
  \draw (0,0) rectangle ++(3,3);
  \draw (1.8,0) -- (1.8,3);
  \node at (2.3,2.5) {$\Gamma_\text{plan}$};
  \node at (0.5,0.3) {$\Omega_1$};  
  \node at (2.5,0.3) {$\Omega_2$};  
  \node at (2.45,1) {$\mu_2=1$};  
\end{tikzpicture}
\end{center}
\subcaption{Planar interface.}
\end{subfigure}
\begin{subfigure}[c]{0.4\textwidth}
\begin{center}
\begin{tikzpicture}
  \draw (0,0) rectangle ++(3,3);
  \draw (1.5,1.5) circle (1cm);
  \node at (2.4,2.5) {$\Gamma_\text{spher}$};
  \node at (1.5,1.5) {$\Omega_1$};    
  \node at (0.5,0.3) {$\Omega_2$};    
  \node at (2.4,0.3) {$\mu_2=1$};    
\end{tikzpicture}
\end{center}
\subcaption{Spherical interface.}
\end{subfigure}
\caption{Sketch of the problem domains.}
\label{fig:domains}
\end{figure}

\subsection{Discretization error} 
For testing the proposed discretizations \eqref{eq:nitscheBfa},  \eqref{eq:nitsche_paramfree_bform} and \eqref{eq:munitsche} we consider the case with $\Gamma_\text{spher}$, with a prescribed solution
\[
u^* = \alpha(\vv{x}) \, ( \| \vv{x} - \vv{m} \|_2^2 - r^2 ), ~~\text{with}~~
\alpha(\vv{x}) := \begin{cases} 
 \mu_2 & \text{for } \vv{x} \in \Omega_1, \\
  \mu_1 & \text{for } \vv{x} \in \Omega_2.
\end{cases}
\]
We use $\mu_2=1$ for all experiments.

Hence, $u^*$ fulfills the equations \eqref{eq:ellipIf1}--\eqref{eq:ellipIf4} when setting $u_D = u^*$ and $f=-6\mu_1\mu_2$.  Table \ref{tab:conv_hansbo_mu2}--\ref{tab:conv_coeff_mu2} show the convergence results for various choices of $\mu_1$ on different refinement levels for the three discretizations. As measure for the convergence we use the $L^2$-error $e(\ell) := \| u_{h_{\ell}} - u^* \|_{0,\Omega}$.

\begin{table}[h]
\vspace{0.5cm}
\begin{tabular}{ r r r r  r r r  r r}
 \toprule
 & \multicolumn{2}{c}{$\mu_1=0.9$}  & \multicolumn{2}{c}{$\mu_1=0.1$} & \multicolumn{2}{c}{$\mu_1=10^{-3}$} & \multicolumn{2}{c}{$\mu_1=10^{-5}$} \\
 \cmidrule(r){2-3} \cmidrule(l){4-5} \cmidrule(l){6-7} \cmidrule(l){8-9}
 \multicolumn{1}{ l }{$\ell$} & \multicolumn{1}{c}{$e(\ell)$} &
\multicolumn{1}{l }{e.o.c.} & \multicolumn{1}{c}{$e(\ell)$} &
\multicolumn{1}{l }{e.o.c.} & \multicolumn{1}{c }{$e(\ell)$} &
\multicolumn{1}{l }{e.o.c.} & \multicolumn{1}{c }{$e(\ell)$} &
    \multicolumn{1}{l }{e.o.c.} \\  \midrule
0&3.27E-01&    &4.83E-02&    &4.38E-02&    &4.41E-02& \\
1&8.20E-02&2.00&1.69E-02&1.51&4.07E-02&0.11&4.42E-02&0.00 \\
2&2.05E-02&2.00&3.98E-03&2.09&2.51E-02&0.69&3.46E-02&0.35 \\
3&5.14E-03&2.00&9.47E-04&2.07&8.58E-03&1.55&2.01E-02&0.78 \\
4&1.28E-03&2.00&2.27E-04&2.06&1.72E-03&2.31&1.05E-02&0.94 \\ \bottomrule
\end{tabular}
\caption{$L^2$-errors for $\uellg$ obtained by \eqref{eq:nitscheBfa} with $\lambda_N=10$ for various choices of $\mu_1$ with corresponding convergence order. Spherical interface.}
\label{tab:conv_hansbo_mu2}
\end{table}

\begin{table}[h]
\begin{tabular}{ r r r r  r r r  r r}
 \toprule
 & \multicolumn{2}{c}{$\mu_1=0.9$}  & \multicolumn{2}{c}{$\mu_1=0.1$} & \multicolumn{2}{c}{$\mu_1=10^{-3}$} & \multicolumn{2}{c}{$\mu_1=10^{-5}$} \\
 \cmidrule(r){2-3} \cmidrule(l){4-5} \cmidrule(l){6-7} \cmidrule(l){8-9}
 \multicolumn{1}{ l }{$\ell$} & \multicolumn{1}{c}{$e(\ell)$} &
\multicolumn{1}{l }{e.o.c.} & \multicolumn{1}{c}{$e(\ell)$} &
\multicolumn{1}{l }{e.o.c.} & \multicolumn{1}{c }{$e(\ell)$} &
\multicolumn{1}{l }{e.o.c.} & \multicolumn{1}{c }{$e(\ell)$} &
    \multicolumn{1}{l }{e.o.c.} \\  \midrule
0&3.26E-01&    &4.40E-02&    &4.21E-02&    &4.41E-02& \\
1&8.19E-02&2.00&1.36E-02&1.69&3.27E-02&0.36&4.37E-02&0.01 \\
2&2.05E-02&2.00&3.45E-03&1.98&1.07E-02&1.60&3.35E-02&0.38 \\
3&5.14E-03&2.00&8.66E-04&2.00&2.16E-03&2.31&1.80E-02&0.89 \\
4&1.28E-03&2.00&2.16E-04&2.00&3.79E-04&2.51&7.29E-03&1.31 \\ \bottomrule
\end{tabular}
\caption{$L^2$-errors for $\uellg$ obtained by \eqref{eq:nitsche_paramfree_bform} for various choices of $\mu_1$ with corresponding convergence order. Spherical interface.}
\label{tab:conv_param_mu2}
\end{table}

\begin{table}[h]
\begin{tabular}{ r r r r  r r r  r r}
 \toprule
 & \multicolumn{2}{c}{$\mu_1=0.9$}  & \multicolumn{2}{c}{$\mu_1=0.1$} & \multicolumn{2}{c}{$\mu_1=10^{-3}$} & \multicolumn{2}{c}{$\mu_1=10^{-5}$} \\
 \cmidrule(r){2-3} \cmidrule(l){4-5} \cmidrule(l){6-7} \cmidrule(l){8-9}
 \multicolumn{1}{ l }{$\ell$} & \multicolumn{1}{c}{$e(\ell)$} &
\multicolumn{1}{l }{e.o.c.} & \multicolumn{1}{c}{$e(\ell)$} &
\multicolumn{1}{l }{e.o.c.} & \multicolumn{1}{c }{$e(\ell)$} &
\multicolumn{1}{l }{e.o.c.} & \multicolumn{1}{c }{$e(\ell)$} &
    \multicolumn{1}{l }{e.o.c.} \\  \midrule
0&3.26E-01&    &4.59E-02&    &2.49E-02&    &2.49E-02& \\
1&8.44E-02&1.95&1.65E-02&1.47&1.32E-02&0.91&1.32E-02&0.91 \\
2&2.08E-02&2.02&4.00E-03&2.05&3.14E-03&2.08&3.14E-03&2.08 \\
3&5.18E-03&2.01&9.47E-04&2.08&7.24E-04&2.12&7.23E-04&2.12 \\
4&1.29E-03&2.01&2.28E-04&2.05&1.70E-04&2.09&1.70E-04&2.09 \\ \bottomrule
\end{tabular}
\caption{$L^2$-errors for $\uellg$ obtained by \eqref{eq:munitsche} with stabilization parameter $\varepsilon_g=0.1$ and $\lambda_N=10$ for various choices of $\mu_1$ with corresponding convergence order. Spherical interface.}
\label{tab:conv_coeff_mu2}
\end{table}

The methods \eqref{eq:nitscheBfa} and \eqref{eq:nitsche_paramfree_bform}  show optimal order of convergence only for modest contrast in the viscosities, whereas \eqref{eq:munitsche} is robust for the whole range (that we considered).
We obtain similar results when varying $\mu_1$ instead of $\mu_2$ and keeping $\mu_2$ fixed, and therefore those tables are not included here.

\begin{remark} \label{Remcondnumber}  \rm For the case $\mu_1=0.1$ we computed the condition numbers on the levels $\ell=3$ and $4$ for the three methods. For the stiffness matrices corresponding to the discretizations \eqref{eq:nitscheBfa}, \eqref{eq:nitsche_paramfree_bform}, \eqref{eq:munitsche} the condition numbers for $\ell=3$ are $4.5 \cdot 10^9$, $4.2 \cdot 10^9$ and $3.7 \cdot 10^3$, and for $\ell=4$ these are $3.1 \cdot 10^{10}$,  $2.9 \cdot 10^{10}$ and $3.9 \cdot 10^3$.  This illustrates the poor conditioning due to small cuts, for the methods without ghost penalty stabilization. It is known from the literature that for linear (cut) finite elements this deterioration can be strongly damped by diagonal scaling of the stiffness matrix. If we consider the diagonally scaled stiffness matrix of the three methods the condition numbers for $\ell=3$ are 
$4.1 \cdot 10^{2}$, $4.9 \cdot  10^{2}$ and $7.6 \cdot 10^2$ respectively, and for $\ell=4$ these are $1.6\cdot 10^3$,  $1.9\cdot 10^3$ and $3.1 \cdot 10^3$. 
\end{remark}

\subsection{Multigrid solver for small coefficient jumps}
We present results of  experiments with the multigrid method introduced in section \ref{sectMG}.  First, we consider the discretizations without additional stabilization, i.e. \eqref{eq:nitscheBfa} and \eqref{eq:nitsche_paramfree_bform}, which, due to a lack of robustness with respect to the interface location, may result  in (very) poorly conditioned system matrices. Every multigrid iteration consists of one V-cycle with 2 pre- and post-smoothing steps of the Gauss-Seidel smoother. The iteration is stopped once a relative residual (in the Euclidean norm) of $10^{-8}$ is reached.
We first consider the problem \eqref{eq:ellipIf} with a planar interface. In this case  there is no error in the interface approximation and $\Gamma_\ell=\Gamma_\text{plan}$  on all levels. Therefore the finite element spaces are nested, i.e. $\Vellmgl \subset \Vellgl$ for $\ell=1,2, \ldots$, and  the prolongation corresponds to simple injection based on the nesting of the spaces.  

We conduct the experiments with $x_\Gamma = 1.321$ for $\Gamma_\text{plan}$, the right-hand-side  $f=xyz$ and boundary data $u_D = 0$ in problem \eqref{eq:ellipIf}. Table \ref{tab:mg_plan_hnitsche} shows the multigrid iteration numbers for \eqref{eq:nitscheBfa} with $\lambda_N=10$ and varying $\mu$-ratios.
\begin{center}
 \vspace{0.5cm}
\begin{tabular}{ r r r r r }
\toprule
\multicolumn{1}{ l }{$\ell$} & \multicolumn{1}{c }{$\mu_1=0.9$} & \multicolumn{1}{c }{$\mu_1=0.5$} & \multicolumn{1}{c }{$\mu_1=0.1$} & \multicolumn{1}{c }{$\mu_1=0.01$} \\ \midrule
1 & 8 & 10 & 31 & 72 \\  
2 & 10 & 10 & 24 & 127 \\  
3 & 11 & 11 & 36 & 297 \\  
4 & 11 & 11 & 25 & 219 \\   \bottomrule
\end{tabular}
\captionof{table}{\eqref{eq:nitscheBfa}: Multigrid iteration numbers for various coefficient ratios and fixed $x_\Gamma=1.321$.}
\label{tab:mg_plan_hnitsche}
 \vspace{0.5cm}
\end{center}

We  observe that for moderate coefficient jumps the multigrid iteration numbers are essentially independent of the discretization level. However, the \eqref{eq:nitscheBfa} discretization requires the penalty parameter $\lambda_N$ which has a strong influence on the iteration numbers. The influence of $\lambda_N$ is illustrated in Table \ref{tab:mg_hnitsche_lam} for a fixed discretization level $\ell=2$ and the moderate coefficient ratio $\mu_1/\mu_2 = 0.5$.

\begin{center}
 \vspace{0.5cm}
\begin{tabular}{ l r r r r r }
\toprule
$\lambda_N$ & 1 & 10 & 20 & 100 & 1000 \\  \midrule 
iterations  & 11 & 10 & 13 & 43 & 143 \\  \bottomrule
\end{tabular}
\captionof{table}{Multigrid iteration numbers for different $\lambda_N$ and fixed $x_\Gamma=1.321$, $\ell=2$ and $\mu_1 = 0.5$.}
\label{tab:mg_hnitsche_lam}
 \vspace{0.5cm}
\end{center}
Due to the strong dependence of the multigrid iteration numbers on $\lambda_N$, we repeat the experiment from Table \ref{tab:mg_plan_hnitsche} but change the discretization from \eqref{eq:nitscheBfa} to \eqref{eq:nitsche_paramfree_bform} where the penalty term is implicitly determined. 

\begin{center}
 \vspace{0.5cm}
\begin{tabular}{ r r r r r }
\toprule
\multicolumn{1}{ l }{$\ell$} & \multicolumn{1}{c }{$\mu_1=0.9$} & \multicolumn{1}{c }{$\mu_1=0.5$} & \multicolumn{1}{c }{$\mu_1=0.1$} & \multicolumn{1}{c }{$\mu_1=0.01$} \\ \midrule
 1 & 8 & 8 & 18 & 93 \\  
2 & 10 & 10 & 12 & 77 \\  
3 & 10 & 10 & 16 & 144 \\  
4 & 11 & 11 & 15 & 99 \\ \bottomrule
\end{tabular}
\captionof{table}{\eqref{eq:nitsche_paramfree_bform}: Multigrid iteration numbers for various coefficient ratios and fixed $x_\Gamma=1.321$.}
\label{tab:mg_plan_pnitsche}
 \vspace{0.5cm}
\end{center}
The iteration numbers in Table \ref{tab:mg_plan_pnitsche} are almost identical to the ones in Table \ref{tab:mg_plan_hnitsche} for the coefficient jumps $\mu_1/\mu_2 \in \{0.9, 0.5\}$ and are  more robust for larger contrast. Due to the benefits of the parameter-free discretization method \eqref{eq:nitsche_paramfree_bform}, in the experiments below for the discretizations \emph{without} stabilization we restrict to this method.

\subsection*{Approximate interface}
In the above experiments the interface $\Gamma$ was linear and hence on the different levels $\ell$  we have $\Gamma_\ell = \Gamma$.
Now we consider the test problem described above with $\Gamma_\text{spher}$.  In this case the interface approximations on different levels vary, i.e., $\Gamma_\lmone \neq \Gamma_\ell$ and therefore the finite element spaces are not nested:  $\Vellmgl \not \subset \Vellgl$. 
Table \ref{tab:mg_circ_if} shows the multigrid iteration numbers for varying coefficient ratios.

\begin{center}
 \vspace{0.5cm}
\begin{tabular}{ r r r r r }
\toprule
\multicolumn{1}{ l }{$\ell$} & \multicolumn{1}{c }{$\mu_1=0.9$} & \multicolumn{1}{c }{$\mu_1=0.5$} & \multicolumn{1}{c }{$\mu_1=0.1$} & \multicolumn{1}{c }{$\mu_1=0.01$} \\ \midrule
1 & 8 & 8 & 18 & 89 \\  
2 & 10 & 10 & 17 & 101 \\  
3 & 11 & 11 & 14 & 88 \\  
4 & 11 & 11 & 16 & 118 \\  \bottomrule
\end{tabular}
\captionof{table}{\eqref{eq:nitsche_paramfree_bform}: Multigrid iteration numbers for various coefficient ratios and $\Gamma_\text{spher}$.}
\label{tab:mg_circ_if}
 \vspace{0.5cm}
\end{center}
We  observe that the results are similar to the exact interface case  in Table \ref{tab:mg_plan_pnitsche}. 
This is due to the use of a canonical prolongation operator $p_\ell: \Vellmgl \to  \Vellgl $. 

To further investigate the multigrid robustness with respect to the interface position, we fix the coefficient ratio $\mu_1/\mu_2 = 0.5$. We vary the interface position for $\Gamma_\text{spher}$ by changing the origin of the circle $\vv{m} = \vv{m}_0 + \delta \vv{e}$ with $\vv{m}_0 = (1.03,1.02,1.01)^T$, $\vv{e} = (1,1,1)^T$ and variable $\delta$. The results for $\ell=1,\ldots,4$ are displayed in Table \ref{tab:mg_iface_pos}.
\begin{center}
 \vspace{0.5cm}
\begin{tabular}{ r r r r r }
\toprule
\multicolumn{1}{ l }{$\ell$} & \multicolumn{1}{c }{$\delta = 0$} & \multicolumn{1}{c }{$\delta = 0.1$} & \multicolumn{1}{c }{$\delta=0.2$} & \multicolumn{1}{c }{$\delta=0.3$} \\ \midrule
1 & 8 & 9 & 8 & 8 \\  
2 & 10 & 10 & 10 & 10 \\  
3 & 11 & 11 & 11 & 11 \\  
4 & 11 & 11 & 11 & 11 \\  \midrule
$\kappa_2(D^{-1}A_4)$ & $1.9 \cdot 10^3$ & $2.0 \cdot 10^3$ & $2.0 \cdot 10^3$ & $2.1 \cdot 10^{3}$ \\
$\kappa_2(A_4)$ & $1.7 \cdot 10^9$ & $9.1 \cdot 10^9$ & $5.0 \cdot 10^{10}$ & $3.0 \cdot 10^{12}$ \\ \bottomrule
\end{tabular}
\captionof{table}{\eqref{eq:nitsche_paramfree_bform}: Multigrid iteration numbers for various interface positions and condition numbers for $\ell=4$. Spherical interface.}
\label{tab:mg_iface_pos}
 \vspace{0.5cm}
\end{center}
The multigrid iteration numbers are (almost) the same for every interface position and  we conclude that the rate of convergence of the method is not sensitive with respect to the interface location. The condition numbers are large and (strongly) depend on $\delta$; cf. last row in Table \ref{tab:mg_iface_pos}. A diagonal scaling of the stiffness matrix leads to an enormous reduction of the condition number (cf. Remark~\ref{Remcondnumber}) and less sensitivity with respect to the interface location.

\subsection{Multigrid solver for large coefficient jumps} \label{sectlarge}
As seen in the previous section, for increasing coefficient ratios the multigrid iteration numbers deteriorate. Therefore, we repeat the experiment from Table \ref{tab:mg_circ_if}, i.e., for the \eqref{eq:nitsche_paramfree_bform} discretization of the test problem with the spherical interface, but instead of a standard Gauss-Seidel smoother we use the GS-IC smoother from algorithm \ref{alg:ifAsSmooth}. Table \ref{tab:mg_exsolve} shows the multigrid iteration numbers. In brackets for small to medium contrast we put the maximum number of diagonally preconditioned CG iterations to reach a relative tolerance of $10^{-2}$ for the solution of $A^\Gamma \vv{x} = R^\Gamma \vv{r}$. For a larger contrast it is more efficient to use a direct solver for the interface matrix, as too many CG iterations would be required (e.g. 496 iterations for $\ell=4$ and $\mu_1=10^{-5}$).

\begin{center}
\vspace{0.5cm}
\begin{tabular}{ r r r r r r r r}
 \toprule
\multicolumn{1}{ l }{$\ell$} & $\mu_1=0.9$ & $0.5$ & $0.1$ &$0.01$ & $0.001$ &  $10^{-4}$ &  $10^{-5}$\\  \midrule
1 & 7 (10)& 7 (10) & 7 (15) &7 (31) & 7 &7 & 7\\
2 & 9 (10)& 9 (10)& 9 (15)&9 (37)& 9 &9 & 9\\
3 & 10 (11)& 10 (11)& 10 (15)&10 (38)& 11&17 & 18\\
4 & 10 (12)& 10 (12)& 10 (15)&10 (42)& 12&21 & 29\\  \bottomrule
\end{tabular}
\captionof{table}{\eqref{eq:nitsche_paramfree_bform}: Multigrid iteration numbers for various coefficient ratios with GS-IC smoother. The maximum number of CG iterations required to solve the system (\texttt{tol}=$10^{-2}$) with $A^\Gamma$ are displayed in brackets. Spherical interface. }
\label{tab:mg_exsolve}
\vspace{0.5cm}
\end{center}
We observe that multigrid iteration numbers are robust for much larger coefficient ratios than with a standard Gauss-Seidel smoother. For small to medium contrast, the iterations required to solve the systems with the interface matrix $A^\Gamma$ are very low and almost constant w.r.t. the refinement level $\ell$. This confirms the uniform (in $\ell$) condition number bounds derived in section~\ref{sectcond}. For larger $\mu$-contrast the multigrid method  is less robust. 

\begin{remark} \rm 
Solving the systems with $A^\Gamma$ approximately with a tolerance of $10^{-2}$ for the relative residual yields (almost) the same multigrid iteration numbers as a direct solve.
\end{remark}

In Table \ref{tab:condnum_agamma} we show  condition numbers of the diagonally scaled interface matrix needed within GS-IC smoother.  These  confirm the theoretical findings given in section \ref{sectlocalA}. For $\mu_1=10^{-2}$ the almost constant small condition number is reflected by almost constant small iteration numbers of the interface CG solver; on the other hand, for  $\mu_1=10^{-4}$, the condition numbers are much larger leading to unacceptably high iteration numbers of the interface CG solver. For such high contrast cases we recommend to use a sparse direct solver.

\begin{center}
\vspace{0.5cm}
\begin{tabular}{ r  r r  r r }
 \toprule
 & \multicolumn{2}{c}{\eqref{eq:nitsche_paramfree_bform}: $\kappa_2(D^{-1}A^\Gamma_\ell)$}  & \multicolumn{2}{c}{\eqref{eq:nitscheBfa}: $\kappa_2(D^{-1}A^\Gamma_\ell)$}  \\
 \cmidrule(r){2-3} \cmidrule(l){4-5} 
 \multicolumn{1}{ l }{$\ell$} & \multicolumn{1}{c}{$\mu_1=0.01$} &
\multicolumn{1}{l }{$\mu_1=10^{-4}$} & \multicolumn{1}{c}{$\mu_1=0.01$} &
\multicolumn{1}{l }{$\mu_1=10^{-4}$}  \\  \midrule
1& 300 & 1343 & 213 & 309     \\
2& 224 & 3533 & 235 & 916 \\
3& 243 & 6420 & 251 & 2904 \\
4& 268 & 8225 & 342 & 6911 \\ \bottomrule
\end{tabular}
\captionof{table}{Condition number of the diagonally scaled interface matrix $A^\Gamma$ in algorithm \ref{alg:ifAsSmooth} for various levels $\ell$ and different $\mu_1$.}
\label{tab:condnum_agamma}
\vspace{0.5cm}
\end{center}
In Table~\ref{tab:fill_in}  we give numbers for the fill-in in the lower triangular matrix  $L_\ell$ of the sparse Cholesky decomposition $A_\ell^\Gamma= L_\ell L_\ell^T$ of the local interface matrix. These results confirm that the direct solver has optimal computational efficiency, cf.~section~\ref{sectdirect}. 
\begin{center}
 \vspace{0.5cm}
\begin{tabular}{ r r r}
\toprule
\multicolumn{1}{ l }{$\ell$} & $ | \mathcal{N}(A_\ell) |$ & $|\mathcal{N}(L_\ell)| / |\mathcal{N}(A_\ell)|$ \\ \midrule
      1 &    6,835 &  0.71  \\
      2 &   56,038 &    0.58  \\
      3 &   466,267 &    0.45  \\
      4 &   3,824,281 &    0.33  \\ \bottomrule
\end{tabular}
\captionof{table}{Number of nonzero entries in $A_\ell$ ($ \mathcal{N}(A_\ell)$) and comparison with number of nonzero entries of sparse Cholesky factor in $A_\ell^\Gamma=L_\ell  L_\ell^T$. }
\label{tab:fill_in}
 \vspace{0.5cm}
\end{center}

\subsubsection*{Coefficient stable and interface stabilized discretization} As seen in the experiments above, the proposed multigrid method is only robust up to a certain contrast in the diffusion coefficient. The deterioration in the multigrid convergence rate may be related to the lack of robustness of the discretization  method \eqref{eq:nitsche_paramfree_bform}.
Clearly, for large coefficient jumps it is better to use a discretization method that is robust with respect to the coefficient ratio, such as \eqref{eq:munitsche}. We apply the multigrid method, with GS-IC smoother, to the linear system resulting from the  discretization method \eqref{eq:munitsche}. Results are shown in Table \ref{tab:mg_circ_cstab}. 

\begin{center}
 \vspace{0.5cm}
\begin{tabular}{ r r r r r r}
\toprule
\multicolumn{1}{ l }{$\ell$} & \multicolumn{1}{c }{$\mu_1=0.9$} & \multicolumn{1}{c }{$\mu_1=0.1$} & \multicolumn{1}{c }{$\mu_1= 10^{-3}$} & \multicolumn{1}{c }{$\mu_1=10^{-5}$} & \multicolumn{1}{c }{$\mu_1=10^{-7}$}\\ \midrule
1 & 7 & 7 & 7 & 7 & 7 \\  
2 & 9 & 9 & 9 & 9  & 9\\  
3 & 10 & 10 & 10 & 10 &10 \\  
4 & 11 & 11 & 11 & 11  & 11\\  \bottomrule
\end{tabular}
\captionof{table}{\eqref{eq:munitsche}: Multigrid iteration numbers with GS-IC smoother for various coefficient ratios, with stabilization parameter $\varepsilon_g = 0.1$ and $\lambda_N=10$. Spherical interface.}
\label{tab:mg_circ_cstab}
 \vspace{0.5cm}
\end{center}
We observe  a very robust behavior of the multigrid method, with iteration numbers that are essentially independent of the value of $\mu_1$.
 Finally we repeat the experiment from Table \ref{tab:mg_hnitsche_lam} for the \eqref{eq:munitsche} discretization with the GS-IC smoother (however for the spherical interface $\Gamma_\text{spher}$ instead of the planar one).  Results are  shown in Table~\ref{tab:mg_cstab_lambda}.

\begin{center}
 \vspace{0.5cm}
\begin{tabular}{ l r r r r r }
\toprule
$\lambda_N$ & 1 & 10 & 20 & 100 & 1000 \\  \midrule 
iterations ($\mu_1=0.1$)  & 12 & 9 & 9 & 9 & 9 \\
iterations ($\mu_1= 10^{-5}$)  & div & 9 & 9 & 9 & 9 \\  \bottomrule
\end{tabular}
\captionof{table}{\eqref{eq:munitsche}: Multigrid iteration numbers with GS-IC smoother for $\Gamma_\text{spher}$ with stabilization parameter $\varepsilon_g = 0.1$ on fixed $\ell=2$ and  for different coefficient ratios and various $\lambda_N$.}
\label{tab:mg_cstab_lambda}
 \vspace{0.5cm}
\end{center}
The method diverges for $\mu_1 = 10^{-5}$ with stabilization parameter value $\lambda_N=1$, which is probably due to the fact that the discretization is unstable for $\lambda_N$ too small. For all other cases, we observe that the multigrid solver is robust with respect to variations in $\lambda_N$. The robustness is due to the use of the (robust) GS-IC smoother. Applying the latter to the \eqref{eq:nitscheBfa} discretization yields similar results.

\begin{remark} \label{RemGalerkinapproach} \rm  We obtain almost identical multigrid iteration numbers for all three discretizations when using the Galerkin approach for the coarse grid matrices $A_{\ell-1}=\br_{\ell}A_{\ell} \bp_{\ell}$. It can be shown that these Galerkin coarse grid matrices have no  extra fill-in  compared to the case of direct discretization.
\end{remark}
 
\section{Conclusions}
We considered the discrete problems resulting from known  unfitted finite element discretizations of the Poisson interface problem. The main topic of the paper is the development of an efficient  multigrid method for this class of discrete problems. A new prolongation operator for the unfitted finite element space was constructed based on the injection of standard finite element spaces on locally extended subdomains. Using level-dependent interface approximations for the discretization of the coarse grid matrices, which can be seen as a kind of quadrature error, did not affect the multigrid convergence. For the case of large contrast in the diffusion coefficients a novel interface smoother has been presented. For small contrast a standard Gauss-Seidel smoothing is already sufficient  to obtain an efficient multigrid method. As a strong dependence of the multigrid convergence rate on the penalty parameter $\lambda_N$ was observed, the parameter free discretization \eqref{eq:nitsche_paramfree_bform} turned out to be preferable. 
Even though the condition number of the resulting stiffness matrix strongly depends on the interface position (for the case without stabilization), the resulting multigrid method did not show any dependence on that. For the case of large coefficient jumps the stabilized discretization \eqref{eq:munitsche} turns out to be better (as known from the literature). The application of the smoother with local interface correction results in  a robust multigrid method. Properties of the interface smoother, concerning conditioning and fill-in in the Cholesky decomposition, are derived. These results imply  optimal arithmetic costs for one multigrid iteration.

\section*{Acknowledgments}
We thank Christoph Lehrenfeld for the fruitful discussions concerning the parameter-free Nitsche discretization. Financial support through the DFG project Re 1461/6-1 is gratefully acknowledged.

\bibliographystyle{siam}
\bibliography{literatur}

\end{document}

%% file: twodom1.tex
\colorlet{poscolor}{rwthBlue}
  \colorlet{negcolor}{rwthGreen}
  \colorlet{pncolor}{IGPMred}
  \noindent 
   \begin{tikzpicture}
    \begin{axis}
    [
     width = \textwidth,
     height= 0.5\textwidth,
     xmin = -0.1,
     xmax = 5,
     clip = false,
     axis lines = middle,
     x axis line style=-,
     hide y axis,
     xticklabels = {},
     line width = 1pt,
     major tick length = 5pt,
     every tick/.style={thick},
     mark size = 1.5pt,
     xtick = {0,1,2,3,4,5,6}
    ]
    \addplot[line width = 2pt, dashed, color = black] coordinates{ (2.8,-0.1) (2.8,2.4) };
        
    
     \addplot[line width = 1.8pt, mark = *, color = poscolor] coordinates { (0,0) (1,1) (2,1.5) (3,1.8) };
     \addplot[line width = 1.8pt, mark = *, color = negcolor] coordinates { (2,2.3) (3,2) (4,1.2) (5,0) };
     \node at (axis cs:1.3,1.9) {\color{poscolor}$u_{H,1}$};
     \node at (axis cs:3.8,2.6) {\color{negcolor}$u_{H,2}$};

     
     \node at (axis cs:3.3,1.1) {\color{black}$\Gamma_\lmone$};
     
    \end{axis}
   \end{tikzpicture} 

%% file: twodom2.tex
\colorlet{poscolor}{rwthBlue}
  \colorlet{negcolor}{rwthGreen}
  \colorlet{pncolor}{IGPMred}
  \noindent 
   \begin{tikzpicture}
    \begin{axis}
    [
     width = \textwidth,
     height= 0.5\textwidth,
     xmin = -0.1,
     xmax = 5,
     clip = false,
     axis lines = middle,
     x axis line style=-,
     hide y axis,
     xticklabels = {},
     line width = 1pt,
     major tick length = 5pt,
     every tick/.style={thick},
     mark size = 1.5pt,
     xtick = {0,1,2,3,4,5,6}
    ]
    \addplot[line width = 2pt, dashed, color = black] coordinates{ (2.3,-0.1) (2.3,2.4) };
        
    
    
      \foreach \x in {0.5,1.5,...,4.5}
      { 
        \addplot[pncolor] coordinates { (\x,-0.1) (\x,0.1)};
      }
    
    
     \addplot[line width = 1.8pt, mark = *, color = poscolor] coordinates { (0,0) (1,1) (2,1.5) };
  	 \addplot[line width = 1.8pt, color = poscolor] coordinates { (2,1.5) (2.5,1.65) };
	 \addplot[line width = 1.5pt, dotted, color = poscolor] coordinates { (2.5,1.65) (3,1.8) };
     \addplot[only marks, mark = *, color = pncolor] coordinates { (0.5,0.5) (1.5,1.25) (2.5,1.65) };
     \addplot[line width = 1.8pt, mark = *, color = negcolor] coordinates { (2,2.3) (3,2) (4,1.2) (5,0) };
     \addplot[only marks, mark = *, color = pncolor] coordinates { (2.5,2.15) (3.5,1.6) (4.5,0.6) };
     \node at (axis cs:1.2,1.9) {\color{poscolor}$u_{h,1}$};
     \node at (axis cs:3.7,2.6) {\color{negcolor}$u_{h,2}$};

     \node at (axis cs:2.7,1.1) {\color{black}$\Gamma_\ell$};
     
    \end{axis}
   \end{tikzpicture} 